\documentclass[microtype]{gtpart}
\usepackage{tocloft}
\usepackage{mathtools}
\usepackage[all,cmtip]{xy}

\usepackage{hyperref}
\hypersetup{
	colorlinks = true,
	linkcolor = {blue},
	urlcolor = {red},
	citecolor = {blue}
}
\usepackage[scr=rsfs]{mathalpha}
\newcommand{\Rect}{\mathbf{R}}

\newcommand{\Flat}{\mathbb{F}}

\newcommand{\Mass}{\mathbb{M}}
\newcommand{\N}{\mathbb{N}}

\newcommand{\Haus}{\mathcal{H}}

\newcommand{\Lip}{\mathrm{Lip}}

\newcommand{\CP}{\mathbb{C}\mathbb{P}(2)}
\newcommand{\dist}{\mathrm{dist}}
\newcommand{\supp}{\mathrm{spt}\hspace{0.01cm}}
\newcommand{\spt}{\mathrm{spt}\hspace{0.01cm}}

\newcommand{\dV}{d_V\kern-1pt}
\newcommand{\dW}{d_W\kern-1pt}

\newcommand\res{\mathop{\hbox{\vrule height 7pt width .3pt depth 0pt
\vrule height .3pt width 5pt depth 0pt}}\nolimits}

\newcommand{\BSOn}{\mathbf{B S O}(n)}

\newcommand{\univob}{\widetilde{\gamma}^n}

\newcommand{\EilenbergML}{\mathbf{K}(\mathbb{Z},n)}

\title{Optimal smooth approximation of integral cycles}

\author{Fredrick Almgren$^{\dagger}$}
\givenname{}
\surname{}
\address{}
\email{}
\urladdr{}

\author{William Browder}
\givenname{}
\surname{}
\address{}
\email{}
\urladdr{}

\author{Gianmarco Caldini}
\givenname{}
\surname{}
\address{}
\email{}
\urladdr{}

\author{Camillo De Lellis}
\givenname{}
\surname{}
\address{}
\email{}
\urladdr{}

\newcommand\blfootnote[1]{%
  \begingroup
  \renewcommand\thefootnote{}\footnote{#1}%
  \addtocounter{footnote}{-1}%
  \endgroup
}
\newtheorem{thm}{Theorem}[section]    % Standard theorem environment
\newtheorem{corollary}[thm]{Corollary}
\newtheorem{lem}[thm]{Lemma}  
\newtheorem{pro}[thm]{Proposition}    % Lemma eand Proposition nvironments with numbering 
\theoremstyle{definition}
\newtheorem{definition}[thm]{Definition}    % Definition environment with 
\newtheorem{question}{Question}
\theoremstyle{definition}
\theoremstyle{remark}
\newtheorem{remark}[thm]{Remark}             % Unnumbered environment for remarks.
\newtheorem{example}[thm]{Example}             % Unnumbered environment for remarks.
\newtheorem{lemma}[thm]{Lemma}             % Unnumbered environment for remarks.
\theoremstyle{definition}
\newtheorem{assumptions}[thm]{Assumption}    % Definition environment with    % Definition environment with 

\usepackage{stmaryrd}

\begin{document}

\begin{abstract}  
In this article we prove that each integral cycle $T$ in an oriented Riemannian manifold $\mathcal{M}$ can be approximated in flat norm by an integral cycle in the same homology class which is a smooth submanifold $\Sigma$ of nearly the same area, up to a singular set of codimension 5. Moreover, if the homology class $\tau$ is representable by a smooth submanifold, then $\Sigma$ can be chosen free of singularities.
\blfootnote{$^{\dagger}$Deceased on February, 5th 1997.}
\end{abstract}

\maketitle
\vspace{-1cm}
\setlength{\cftbeforesecskip}{9pt}
\tableofcontents

%%%%%%%%%%%%%%%%%%%%   Start of main body of article

\section{Introduction}

\subsection{Motivation and historical background}
Integral currents represent one of the most satisfactory analytic and topological formulations of the concept of \textit{generalized surfaces}, that is $m$-dimensional submanifolds in $(m+n)$-dimensional ambient manifolds having sufficient compactness properties to allow the application of the direct methods in the calculus of variations. Integral currents were introduced by Federer and Fleming in their celebrated article \cite{FedererFleming60} to provide a successful solution to the so-called \textit{oriented} Plateau problem: the problem of finding an oriented generalized surface of smallest area spanning a given boundary or representing a given homology class. One of the main results of \cite{FedererFleming60} is that each homology class with coefficients in $\mathbb{Z}$ can be represented by a cycle of least area. These existence theorems have been followed over the years by powerful regularity theories for minimizers, showing that solutions are \textit{a posteriori} much more regular than one might expect \textit{a priori}, see \cite{DeGiorgi61, Almgren66, Allard72, Almgren83, Almgren00, dlsQ, dls1, dlssns, dls2, dls3}.% This gave rise to the modern field of \emph{geometric calculus of variations.}

The natural question of how much smoother integral currents are with respect to their original definition goes back to the late 1950s and to the origin of the theory of integral currents with the seminal article of Federer and Fleming, see \cite{Federer59, FedererFleming60}. In particular, in \cite[page 1, lines 30-31]{FedererFleming60}, the authors write: {\it``Integral currents are actually much smoother than one might expect from the preceding definition''}, introducing the well-known \emph{deformation theorem} of integral currents and the \emph{strong approximation theorem} by means of polyhedral chains with integer coefficients, see \cite[Theorems 5.5 and 8.22]{FedererFleming60}. The deformation theorem represents a cornerstone in the theory, showing that the space of integral currents is the closure with respect to the flat topology of the space of polyhedral chains with integer coefficients.

A basic question in the theory of integral currents is thus the following.

\begin{question}
{\it ``How closely can one approximate an integral current $T$ representing a given homology class $\tau$ by a smooth submanifold?''}
\end{question}

It may happen, in full generality, that integral currents are singular due to topological obstructions: in \cite{Thom54}, Thom provides an example of a homology class of dimension 7 in a manifold of dimension 14 which is not realizable by means of a submanifold, \textit{cfr.} Example \ref{ex:ThomLensSpaces}. Moreover, it turns out that for each dimension greater than 7 there exist (in some manifold of arbitrarily large dimension) nonrealizable integral homology classes, see  \cite[Théorème III.9]{Thom54}; therefore, any integral current representing such a class must have singularities. Nevertheless, these obstructions motivate the following very natural question. 

\begin{question}\label{q:Lavrentiev}
{\it ``Suppose that a given homology class $\tau$ is realizable by a smooth submanifold, is it always possible to approximate any integral current $T$ representing $\tau$ (and hence, \emph{a fortiori}, any $T$ which is area-minimizing) by smooth submanifolds?''}
\end{question}

The following theorem, which is the focus of this article, provides answers to both questions.

\begin{thm}[Optimal smooth approximation]\label{t:1}
Let $\mathcal{M}$ be a connected smooth closed oriented Riemannian manifold of dimension $m+n$. Let $\varepsilon > 0$, $\tau$ be a fixed nonzero element of the $m$-dimensional integral homology group $\mathbf{H}_{m}(\mathcal{M}, \mathbb{Z})$, and $T$ be an integral cycle representing $\tau$. Then, there is a smooth triangulation $\mathcal{K}$ of $\mathcal{M}$ and an oriented $m$-dimensional smooth submanifold $\Sigma$ of $\mathcal{M}\setminus \mathcal{K}^{m-5}$ (where $\mathcal{K}^{m-5}$ denotes $(m-5)$-skeleton of $\mathcal{K}$) with the following properties.
\begin{enumerate}
\item The $m$-dimensional volume of $\Sigma$ does not exceed the mass of $T$ by more than $\varepsilon$, that is $\mathcal{H}^{m}(\Sigma) \leq \mathbb{M}(T) + \varepsilon$.
\item The current $\llbracket \Sigma \rrbracket$ is an integral cycle homologous to $T$ and there is an integral $(m+1)$-dimensional current $S$ in $\mathcal{M}$ such that $\partial S=\llbracket \Sigma \rrbracket - T$ and $\mathbb{M}(S)< \varepsilon$.
\item If $\tau$ admits a smooth representative, then $\Sigma$ can be chosen to be a smooth submanifold of $\mathcal{M}$.
\end{enumerate}
\end{thm}

\begin{remark}
The codimension 5 construction in Theorem \ref{t:1} is optimal, as shown by the \emph{innately singular} homology class discovered by Thom, see Theorem \ref{t:Thom_innatelysingular}.
\end{remark}

In 1988 Almgren posed these basic questions formally, announcing Theorem \ref{t:1} together with the second named author few years later, see \cite[page 20, line 9]{AlmgrenBrowder88-91} and \cite[page 44, line 21]{Almgren90-93}; nevertheless, the program was never completed and a proof of the announced result never appeared. In \cite{AlmgrenBrowder88-91}, among other things, the authors hint at the strategy of using Thom's criterion in the context of homotopy classes of mappings from $\mathcal{M}$ (less a skeleton) to the Thom complex $T(\widetilde{\gamma}^n)$; building upon these fruitful unpublished ideas this article provides a complete proof of Theorem \ref{t:1}.

\subsection{Consequences of the main theorem}

We outline here several consequences of Theorem \ref{t:1}. Our underlying assumption for all the theorems of this paper is the following\footnote{We refer to Section \ref{s:legenda} for a list of notations.}.

\begin{assumptions}\label{a:1}
$n,m \in \N\setminus \{0\}$ are arbitrary positive integers, $\mathcal{M}$ is a connected smooth oriented closed Riemannian manifold of dimension $m+n$, $\tau$ is a nonzero element of the $m$-dimensional integral homology group $\mathbf{H}_{m}(\mathcal{M}, \mathbb{Z})$ and $T$ is an integral current (hence a cycle) representing $\tau$.
\end{assumptions}

 When dealing with smooth triangulations of $\mathcal{M}$ we will tacitly assume to have fixed some simplicial complex $\mathcal{K}$ together with a piecewise smooth map $t:|\mathcal{K}| \rightarrow \mathcal{M}$. In order to keep our notation simpler, with a slight abuse we will in fact mostly avoid referring to the map $t$ and we will use directly $\mathcal{K}$ also for the smooth triangulation of $\mathcal{M}$; thus, a simplex of the triangulation $\mathcal{K}$ will mean the $t$-image of a simplex of $|\mathcal{K}|$. $\mathcal{K}^j$ will denote the $j$-dimensional skeleton of $\mathcal{K}$, i.e. the union of all $j$-dimensional simplices of $\mathcal{K}$. Moreover, the letter $\Sigma$ will be reserved to denote either smooth oriented $m$-dimensional closed embedded submanifolds of $\mathcal{M}$ or smooth oriented $m$-dimensional embedded submanifolds of $\mathcal{M}\setminus \mathcal{K}^j$ (for some integer $j$) whose topological closure is contained in $\mathcal{K}^j$. 

By Nash’s isometric embedding theorem we consider $\mathcal{M}$ as a submanifold of some Euclidean space $\mathbb{R}^N$ and, by a classical theorem of Whitney, smooth submanifolds of $\mathbb{R}^N$ are smooth retractions of some open neighborhood, see \cite[3.1.19]{Federerbook}. Hence smooth compact submanifolds are Lipschitz neighborhood retracts. For every $k=0,\dots,m+n$, we denote by $\mathcal{Z}_k(\mathcal{M})$ the set of $k$-dimensional integral cycles with support in $\mathcal{M}$ and by $\mathcal{Z}_{k, Lip}(\mathcal{M})$ the set of $k$-dimensional integer Lipschitz cycles with support in $\mathcal{M}$, that is the set of cycles of the form $f_{\#}(P)$ where $f: \mathbb{R}^N \rightarrow \mathcal{M}$ is a Lipschitz map and $P$ is an integer polyhedral cycle in $\mathbb{R}^N$. 

We refer to Section \ref{s:notationandpreliminaries} for the relevant definitions. 

\begin{definition}[\emph{Smooth representability}]\label{d:smoothrepresentability}
Let $\mathcal{M}, \tau$ and $\Sigma$ be as in Assumption \ref{a:1}. We say that $\tau$ is \emph{representable by a smooth submanifold} (or that $\tau$ \emph{admits a smooth representative}) if there exists a smooth embedding $f:\Sigma \rightarrow \mathcal{M}$ such that the fundamental class of $\Sigma$ equals $\tau$, that is $f_*[\Sigma]=\tau$.
\end{definition}

The first consequence of Theorem \ref{t:1} is the absence of the so-called \emph{Lavrentiev gap phenomenon} for the homological Plateau problem\footnote{The Lavrentiev gap phenomenon holds for a functional in the calculus of variations when it has different infima depending on whether the infimum is taken over the whole class of admissible objects or over some smaller class of \emph{more regular} objects. The first of such examples was discovered by Lavrentiev in 1927, see \cite{Lavrentiev27} and
\textit{cfr.} \cite{ButtazzoBelloni95}.}. 
%In the geometric variational theory of integral currents, such a question has been open since its origins in the late 1950s, see \cite{Federer59, FedererFleming60}. In particular, it was not known whether such a Lavrentiev gap phenomenon was present between minimizing sequences of (Lipschitz) polyhedral cycles and minimizing sequences of smooth submanifolds. As a consequence of Theorem \ref{t:1}, we also settle this question in the negative with the following result.

\begin{thm}[Absence of Lavrentiev gaps]\label{t:Lavrentiev}
Let $\mathcal{M}, \tau, T$ and $\Sigma$ as in Assumption \ref{a:1}, and define the following quantities:
\begin{align*}
&\mathbb{M}_{T}:=\min\{\mathbb{M}(T) : T \in \mathcal{Z}_k(\mathcal{M}) \cap \tau\},\\
&\mathbb{M}_{P}:=\inf\{\mathbb{M}(P) : P \in \mathcal{Z}_{k, Lip}(\mathcal{M}) \cap \tau\},\\
&\mathbb{M}_{\Sigma}:=\inf\{\text{Vol}\,^m(\Sigma) : \llbracket\Sigma\rrbracket \in \tau \textrm{ and $\Sigma$ is smooth in $\mathcal{M} \setminus \mathcal{K}^{m-5}$ for some triangulation $\mathcal{K}$}\},\\
&\mathbb{M}_{{\rm Reg}} := \inf\{\text{Vol}\,^m(\Sigma) : \llbracket\Sigma\rrbracket \in \tau \textrm{ and $\Sigma$ is smooth in $\mathcal{M}$}\}\, .
\end{align*}
Then, $\mathbb{M}_{T}= \mathbb{M}_{P}= \mathbb{M}_{\Sigma}$ and, moreover, $\mathbb{M}_T = \mathbb{M}_{{\rm Reg}}$ when $\tau$ is representable by a smooth submanifold.
\end{thm}

\begin{remark}\label{r:St_pi} Note that any class $\tau$ is representable by a smooth submanifold when $n \in \{1,2\}$ or when $m \in \{1,2,3,4,5,6\}$, see Lemma \ref{l:m+4_equivalenza} and Remark \ref{r:rappresentofinomleq4}. In these cases, Theorem \ref{t:1} implies that singularities of integral cycles can be \emph{resolved}, in the sense that they can be approximated by smooth submanifolds.

The request that $\tau$ is representable by a smooth submanifold can often be expressed in terms of the vanishing of some suitable \emph{obstruction}, which is represented by a cohomology operation; in particular, denoting $x$ the Poincar\'e dual of $\tau$ and by $p$ an odd prime, a necessary (and, in some particular dimensions, also sufficient) condition for an integral homology class $\tau$ to be representable by a smooth submanifold is that all $St_p^{2r(p-1)+1}x$ are null, see \cite[Théorème II.20]{Thom54}. We recall that, following Thom's notation for the Bockstein reduced $p^{th}$ powers reduction mod $p$, $St_p^{2r(p-1)+1}$ represent (up to a sign) the following cohomology operations: $$St_p^{2r(p-1)+1}=\beta^* \circ \mathcal{P}^r_p \circ \theta_p:\mathbf{H}^*(X,\mathbb{Z}) \rightarrow \mathbf{H}^{*+2r(p-1)+1}(X,\mathbb{Z}),$$ where $\mathcal{P}^r_p$ is the reduced Steenrod $p^{th}$ power, $\beta^*: \mathbf{H}^*(X,\mathbb{Z}_p) \rightarrow \mathbf{H}^{*+1}(X,\mathbb{Z})$ is the Bockstein associated to the short exact sequence $$0 \rightarrow \mathbb{Z} \rightarrow \mathbb{Z} 
 \rightarrow \mathbb{Z}_p \rightarrow 0$$ and $\theta_p$ the reduction mod $p$ such that $\beta_p=\theta_p \circ \beta^*,$ with $\beta_p: \mathbf{H}^*(X,\mathbb{Z}_p) \rightarrow \mathbf{H}^{*+1}(X,\mathbb{Z}_p)$ the Bockstein associated to $0 \rightarrow \mathbb{Z}_p \rightarrow \mathbb{Z}_{p^2} \rightarrow \mathbb{Z}_p \rightarrow 0.$ 
 
 This necessary condition is a consequence of the fact that the cohomology for odd primes of the oriented Grassmannians is concentrated in dimensions which are multiples of 4.

%$St^{2i(p-1)+1}_p$ & Thom's integral Steenrod power for odd prime $p$ and of dimension $i$, \textit{i.e.} $\beta \mathcal{P}^i_p \mu_p$;\\

%that is $\beta \mathcal{P}^i_p \mu_p$, where $\beta = $ and $\mu_p$ is the reduction mod $p$ $\mathbf{H}^*(X,\mathbb{Z} \rightarrow \mathbf{H}^*(X,\mathbb{Z}_p $

\end{remark}

As a simple corollary of Theorem \ref{t:1} we further deduce the following approximation theorem with cycles of prescribed singularities.

\begin{thm}[Approximation by cycles with prescribed singular sets]\label{t:prescribed}
Let $\mathcal{M}, \tau$ and $T$ be as in Assumption \ref{a:1}. Then there is a sequence of smooth triangulations $\mathcal{K}_j$ of $\mathcal{M}$ and a sequence of smooth embedded oriented $m$-dimensional submanifolds $(\Sigma_j)_j$ in $\mathcal{M}\setminus \mathcal{K}_j^{m-5}$ such that
\begin{itemize}
\item[(a)] $\llbracket \Sigma_j\rrbracket \rightarrow T$ in the sense of currents,
\item[(b)] $\lim_{j\rightarrow \infty}\mathcal{H}^{m}(\Sigma_j) = \mathbb{M}(T)$,
\item[(c)]$\partial \llbracket \Sigma_j\rrbracket =0$ and $\llbracket \Sigma_j\rrbracket$ is in the same homology class as $T$.
\end{itemize}
\end{thm}

In fact conclusion (c) is a simple consequence of the Federer flatness criterion: since $\partial \llbracket \Sigma_j \rrbracket$ is a flat current supported in a set which has $\mathcal{H}^{m-1}$-zero measure, it must be $0$, see \cite[4.1.20]{Federerbook}; therefore it also follows from the convergence to $T$ that $\llbracket \Sigma_j \rrbracket$ is in the same homology class as $T$ for every $j$ sufficiently large.

%\begin{remark}\label{r:novicinanzasupporti} It is not always possible to require that the approximating submanifold $\Sigma$ is (locally) close to the support of $T$, or to the support of any area-minimizing cycle representing $\tau$, \textit{cfr.} Example \ref{e:novicinanzalocale}. \end{remark}

%\begin{thm}[Setwise approximation by cycles with prescribed singular sets]\label{t:supports}
%Let $\mathcal{M}, \tau, T$ and $\Sigma$ be as in Assumptions \ref{a:1}. Assume further that $\supp(T)$ is $m$-dimensional. Let $B \subset \mathcal{M}$ be any subset such that $\tau$ is representable by a smooth submanifold $\Sigma$ in $U_{\varepsilon}(\spt(T)\setminus B)$, for any $\varepsilon>0$. Then there exists a sequence $(T_j)_j$ of integral cycles and a sequence of $\varepsilon_j>0$ such that for every $j$ we have:
%\begin{enumerate}
%\item $\spt(T_j) \subset U_{\varepsilon_j}(\spt(T))$,
%\item $\emph{Sing}(T_j) \subset B$,
%\item $T_j \rightarrow T$ in the sense of currents,
%\item $\lim_{j\rightarrow \infty}\mathbb{M}(T_j) = \mathbb{M}(T).$
%\end{enumerate}
%\end{thm}

%\begin{remark}
%For a polyhedral (integral?) current such a $B$ always exists, namely the $m-4$ skeleton $(m-5)?$. Examples show that for area-minimizing $T$ the set $B$ can be smaller, e.g. the complex variety $\{z^2=w^3\}$, which has a singularity that can be approximated by a smooth submanifold.}
%\end{remark}

\subsection{Overview of the proof}

The main idea of our study is to combine Federer and Fleming's theory of integral currents with tools and techniques from cobordism and homotopy theory.

The proof of Theorem \ref{t:1} can be \textit{grosso modo} descrived as follows. Starting from an $m$-dimensional integral cycle $T$ in a nontrivial homology class $\tau$ of $\mathcal{M}$, we first develop a delicate approximation theorem by means of a cycle $P'$ which is a smooth submanifold outside of a (small) $\delta$-neighborhood $B_\delta$ of the $(m-2)$-skeleton of some triangulation of $\mathcal{M}$, \textit{cfr.} Proposition \ref{p:poly_approx_prescribedsing}. In particular, a subset of the smooth part of $P'$ is a compact smooth submanifold with boundary embedded in a compact manifold with boundary, which we denote by $\Omega$ (ideally we would define $\Omega$ as $\mathcal{M}\setminus B_\delta$, but the latter does not have a smooth boundary: we will get around this technical obstruction by a standard regularization procedure, \textit{cfr.} Section \ref{s:smoothingneighborhood}): this object represents a relative homology class in $\mathbf{H}_m(\Omega,\partial \Omega, \mathbb{Z})$; this will induce, by Theorem \ref{t:Thomboundary} a map $$F:\Omega \rightarrow T(\widetilde{\gamma}^n)$$ with values in the Thom space of the universal oriented $n$-plane bundle and such that the pull-back of the Thom class equals the Poincaré dual of $\tau$, when restricted to $\Omega$. This is known as the (relative) \emph{Thom construction}.

Then, denoting by $Q$ the complement of a small neighborhood $U_\delta$ of the $(m-5)$-skeleton of the triangulation of $\mathcal{M}$, we note that $Q$ has the homotopy type of an $(n+4)$-dimensional skeleton of $\mathcal{M}$, \textit{cfr.} Lemma \ref{l:m-5/n+4}. Thus, we exploit the $n+4$-equivalence between $T(\widetilde{\gamma}^n)$ and the Eilenberg-MacLane space $\mathbf{K}(\mathbb{Z},n)$, \textit{cfr.} Lemma \ref{l:m+4_equivalenza}, to prove that the restriction of the Poincaré dual of $\tau$ to $Q$ admits a lift \begin{equation}\label{e:riassuntof}f:Q \rightarrow T(\widetilde{\gamma}^n)\end{equation} pulling back the Thom class to itself. Applying Theorem \ref{t:Thomboundary} and after some technicalities, this provides an integral cycle $R$ homologous to $\tau$ which is a closed smooth embedded submanifold with singularities all contained in the $(m-5)$-dimensional skeleton $\mathcal{K}^{m-5}$ of $\mathcal{M}$.

Since, by Lemma \ref{l:m-5/n+4}, $\Omega$ has the homotopy type of an $(n+1)$-dimensional complex, we observe that homotopy classes of maps defined on $\Omega$ and with values in $T(\widetilde{\gamma}^n)$ are in one-to-one correspondence with those with values in $\mathbf{K}(\mathbb{Z},n)$, \textit{cfr.} Corollary \ref{c:Whitehead_Homotopyclasses}. This allows us to conclude that the smooth part of $P'$ coincides, up to a homotopy, with the smooth part of $R$ once restricted to $\Omega$. 

The conclusion then follows from a technical geometric measure theory construction, \textit{cfr.} Proposition \ref{p:squash}, saying that if two $m$-dimensional integral cycles $P'$ and $R$ agree outside of a sufficiently small neighborhood of the $(m-2)$-skeleton of the triangulation of $\mathcal{M}$, then we can find a smooth deformation $R'$ of $R$ which is almost coinciding with $R$ and with mass close to the mass of $P'$. This provides the desired approximation $R'$ of Theorem $\ref{t:1}$, satisfying $(1)$ and $(2)$.

Finally, part $(3)$ of Theorem $\ref{t:1}$ is proved following the same lines: under the additional assumption that $\tau$ is representable by a smooth submanifold we immediately obtain a map $$g:\mathcal{M}\rightarrow T(\widetilde{\gamma}^n)$$ which pulls-back the Thom class to the Poincaré dual of $\tau$; the analogous construction can thus be performed just by replacing the map $f$ in \eqref{e:riassuntof} with $g$.

The rest of the paper is organized as follows. In Section \ref{s:notationandpreliminaries} we briefly recall the main notation in the theory of integral currents and some preliminary results in homotopy theory and cobordism. In Section \ref{s:smoothingneighborhood} we collect some technical preliminary lemmas about neighborhoods of skeleta and maps associated to them. In Section \ref{s:GMT} we will prove the two main technical propositions from geometric measure theory: Proposition \ref{p:poly_approx_prescribedsing} and Proposition \ref{p:squash}; an appendix with some elementary facts about triangulations and simplicial decompositions is listed in Section \ref{s:Appendix_triangulations}. Section \ref{s:proof} is dedicated to the proof of Theorem \ref{t:1} and Section \ref{s:optimality} shows the optimality of the construction. %Finally, Section \ref{s:corollaries} contains the proofs of Theorems \ref{t:Lavrentiev} and \ref{t:prescribed}, deduced as simple corollaries of the main theorem.
We add at the end another brief appendix, see Section \ref{s:cohomologyoperations}, recalling some introductory results about cohomology operations and characteristic classes, useful in the proof of Lemma \ref{l:m+4_equivalenza}.

\subsection*{Acknowledgements}
We are particularly grateful to Jacob Lurie and Dennis Sullivan for fruitful conversations. In addition, G.C. would like to thank Guido De Philippis, Mark Grant and Antonio Lerario for useful discussions, and the Institute for Advanced Study, whose warm hospitality is gratefully acknowledged. W.B. and G.C. would also like to thank Candace McCoy. The research of G.C. has been supported by the Associazione Amici di Claudio Dematt\'e.

\section{Notation and preliminary results}\label{s:notationandpreliminaries}
In this section we recall the main definitions and relevant notation. %For a complete treatment of these topics, the reader is referred to \cite{MilnorStasheff, MunkresElements, MosherTangora} and to \cite{Federerbook, Simonbook}.

\subsection{Theory of integral currents}
We briefly recall the main notions of Federer and Fleming's theory of integral currents, \textit{cfr.} also \cite{FedererFleming60, Federerbook, Simonbook}.
The space of $k$-dimensional De Rham currents in $\R^N$ (i.e. continuous linear functionals on the space $\mathcal{D}^k(\R^N)$ of smooth and compactly supported differential $k$-forms in $\R^N$) is denoted by $\mathcal{D}_k(\R^N)$. The boundary of $T\in\mathcal{D}_k(\R^N)$ is defined enforcing Stokes' theorem, namely $\partial T(\varphi)=T(d\varphi)$, and if $\partial T=0$ then $T$ is called a cycle. The mass of $T$ is denoted by $\Mass(T)$ and is defined as the supremum of $T(\omega)$ over all forms $\omega$ with $|\omega(x)| \leq 1$ for all $x \in \R^N$, where $|\cdot|$ denotes an appropriately defined norm called comass. The support of $T$, denoted $\supp(T)$, is the intersection of all closed sets $C$ in $\R^N$ such that $T(\omega)=0$ whenever $\omega \equiv 0$ on $C$. For every compact Lipschitz neighborhood retract $M \subset \R^N$, we will denote by $\mathcal{D}_k(M)$ the set $$\mathcal{D}_k(M):=\{T \in \mathcal{D}_k(\R^N) \mid \spt(T) \subset M\}.$$

We recall that a current $T\in\mathcal{D}_k(\mathbb{R}^N)$ is integer rectifiable (and we write $T\in\Rect_k(\mathbb{R}^N)$) if we can identify $T$ with a triple $(E,\tau,\theta)$, where $E\subset K$ is a $k$-rectifiable set, $\tau(x)$ is a unit $k$-vector spanning the tangent space $T_xE$ at $\mathcal{H}^k$-a.e. $x$ and $\theta\in L^1(\mathcal{H}^k \res E, \mathbb{Z})$ is an integer-valued multiplicity. The identification means that the action of $T$ can be expressed by 
\begin{equation}\label{e:rectifiablecurrent} 
T(\omega)=\int_E\langle\omega(x),\tau(x)\rangle \, \theta(x) \, d\Haus^k(x), \quad \mbox{ for every $\omega\in \mathcal{D}^k(\R^N)$}.\end{equation} 
If $T$ is as in \eqref{e:rectifiablecurrent}, we denote it by $T=\llbracket E, \tau, \theta \rrbracket$. We will often use the shorthand notation $T=\theta \llbracket E \rrbracket$ if $\theta$ is constant and the orientation is clear from the context. We denote by $\mathbf{I}_k(\R^N)$ the subgroup of $k$-dimensional integral currents, that is the set of currents $T\in\Rect_k(\R^N)$ with $\partial T\in\Rect_{k-1}(\R^N)$. If $T=\llbracket E, \tau, \theta \rrbracket\in\Rect_k(\R^N)$ and $B \subset \mathbb{R}^N$ is a Borel set, we denote the restriction of $T$ to $B$ by setting $T\res B:=\llbracket E \cap B, \tau, \theta \rrbracket.$ The set of integer rectifiable (respectively integral) $k$-currents with support in a compact Lipschitz neighborhood retract $M$ is denoted by $\Rect_k(M)$ (respectively $\mathbf{I}_k(M)$). We denote by $\mathcal{Z}_k(M)$ the space  of integral cycles with support in $M$, \textit{i.e.} the space of integral currents $T \in \mathbf{I}_k(M)$ with $\partial T=0$.

We recall that the \emph{(integral) flat norm} $\Flat(T)$ of an integral current $T \in \mathbf{I}_k(M)$ is defined by: \begin{equation}\label{e:flatdef} \Flat(T) := \min \{ \Mass(R) + \Mass(S) \mid T=R +\partial S, \, R \in \mathbf{I}_k(M),\, S \in \mathbf{I}_{k+1}(M)\}.\end{equation}

Given a smooth, proper map $f : \R^N \rightarrow \R^{N'}$ and a $k$-current $T$ in $\R^N$, the \emph{push-forward} of $T$ according to the map $f$ is the $k$-current $f_{\sharp}T$ in $\R^{N'}$ defined by
\begin{equation}\label{d:push-forward}
f_{\sharp}T (\omega) := T(f^{*}\omega), \quad\mbox{for every $\omega\in\mathcal{D}^k(\R^{N'})$},\end{equation} 
where $f^{*}\omega$ denotes the pullback of $\omega$ through $f$. If $T$ is such that $\mathbb{M}(T), \mathbb{M}(\partial T)< \infty$ and $f:\mathbb{R}^N\rightarrow \mathbb{R}^{N'}$ a Lipschitz map such that $f_{|\text{spt}(T)}$ is proper, then the pushforward of $T$ via $f$ can be defined as follows. Let $\varphi \in C_c^\infty(\mathbb{R}^N)$ be a standard mollifier, denote $\varphi_\tau(x):=\tau^{-n}\varphi(\tau^{-1}x)$, for $\tau>0$, and let $f_\tau:= f*\varphi_\tau$ be the smoothing of $f$. The pushforward of $T$ via $f$ is defined as $$f_\sharp T(\omega):=\lim_{\tau \rightarrow 0}f_{\tau\sharp}T(\omega), \quad\text{for every } \omega \in \mathcal{D}^k(\mathbb{R}^{N'}).$$

A $k$-dimensional \emph{polyhedral}\footnote{A more appropriate term might be ``integral polyhedral'', allowing ``polyhedral'' to have also real coefficients. In this paper the coefficients will always be integers.} current (or \emph{polyhedral chain}) is a current $P$ of the form  \begin{equation}\label{e:poly} P:=\sum_{i=1}^d\theta_i\llbracket \sigma_i\rrbracket, \end{equation} where $\theta_i\in \mathbb N$, $\sigma_i$ are $k$-dimensional simplices in $\R^N$, oriented by (constant) $k$-vectors $n_i$ and $\llbracket \sigma_i \rrbracket=\llbracket \sigma_i, n_i, 1 \rrbracket$ is the multiplicity-one current naturally associated to $\sigma_i$. The subgroup of $k$-dimensional integer polyhedral currents in $\mathbb{R}^N$ will be denoted by $\mathscr{P}_k(\mathbb{R}^N)$, while $\mathcal{Z}_{k,Lip}(M)$ will be used for the set of $k$-dimensional integer Lipschitz cycles with support in $M$, that is the set of cycles of the form $f_{\sharp}(P)$ where $f: \mathbb{R}^N \rightarrow M$ is a Lipschitz map and $P \in \mathscr{P}_k(\mathbb{R}^N)$.

\subsection{Homotopy theory and cobordism}
We briefly recall the main topological notions that will be used later, we refer also to \cite{Spanier, Switzer, Thom54, MilnorStasheff}.

The \emph{mapping cylinder} $M_f$ of a continuous map $f: X \rightarrow Y$ is the quotient space formed from the disjoint union $(X \times [0,1]) \sqcup Y$ by identifying, for each $x \in X$, the point $(x, 1)$ with $f(x) \in Y$; it contains $X \times\{0\}$ as a subspace and has $Y$ as a deformation retract. 

%\begin{remark}\label{r:mapping-cylinder-neigh}
%Fix a (finite dimensional) $CW$-complex $X$ and a closed subset $Y$ for which the complement $X\setminus Y$ is a manifold. Then a \emph{mapping cylinder neighborhood} of $Y$ in $X$ is a closed neighborhood $N \supset Y$ with boundary $\partial N$ a submanifold of $X\setminus Y$, together with a map $f:\partial N \rightarrow Y$ and a homeomorphism of the mapping cylinder of $f$ with $N$ which is the identity on $\partial N$ and $Y$. If the submanifold of $X \setminus Y$ is smooth, then the mapping cylinder $M_f$ can be chosen to be smooth (in the sense that $M_f\setminus Y$ is), \textit{cfr.} \cite{Quinn}. We can use this fact to close an argument in our proof, but we will also give an alternative argument using geometric measure theory. 
%\end{remark}

For $n \geq 1$ and an abelian group $\pi$, the \emph{Eilenberg-MacLane space} $\mathbf{K}(\pi, n)$ is a space with the homotopy type of a $CW$-compex such that $\pi_i(\mathbf{K}(\pi,n))$ vanishes for $i \neq n$ and $\pi_n(\mathbf{K}(\pi,n))\simeq \pi$, where $\pi_i (X)$ denotes the $i$-th homotopy group of the topological space $X$. Recall that the Hopf homotopy classification theorem states that for a (connected) $CW$-complex $X$, an abelian group $\pi$ and for every $n \in \mathbb{N}\setminus \{0\}$ there is a natural isomorphism $$T: [X, \mathbf{K}(\pi, n)] \rightarrow \mathbf{H}^n(X,\pi),$$ where $[X, \mathbf{K}(\pi, n)]$ represents the set of (unbasedpointed) homotopy classes of continuous maps from $X$ to $\mathbf{K}(\pi, n)$ and $\mathbf{H}^n (X, \pi)$ is the $n$-th cohomology group of $X$ with coefficients in $\pi$. The isomorphism has the form $T([f]) = f^*(\iota)$, for a certain $\iota \in \mathbf{H}^n(\mathbf{K}(\pi, n),\pi)$ called the \emph{fundamental class}; $\mathbf{K}(\pi, n)$ is therefore the classifying space of $n$-dimensional cohomology with coefficients in $\pi$. This determines $\mathbf{K}(\pi,n)$ up to homotopy equivalence: that is, the homotopy type of $\mathbf{K}(\pi,n)$ is determined by $\pi$ and $n$, and the identity map of $\pi$ determines, up to homotopy, a canonical homotopy equivalence between any two copies of $\mathbf{K}(\pi,n)$.

We recall that a continuous map $f:X \rightarrow Y$ between path-connected\footnote{We restrict to the case $X$ and $Y$ are path-connected since we only need this in the sequel.} $CW$-complexes is called an \emph{$n$-equivalence} for $n\geq1$ if the induced homomorphism $$f_* : \pi_i(X) \rightarrow \pi_i(Y)$$ is an isomorphism for $0<i<n$ and an epimorphism for $i=n$.

We also recall that %if $f:X \rightarrow Y$  is a continuous map and $M_f$ is the mapping cylinder of $f$, then $f=r \circ i$, where $r:M_f\rightarrow Y$ is a deformation retract; hence,
$f:X \rightarrow Y$ is an $n$-equivalence if and only if the inclusion $i:X \rightarrow M_f$ is an $n$-equivalence. From the long exact sequence of relative homotopy groups, it follows that $i$ is an $n$-equivalence if and only if the relative homotopy group $\pi_i(M_f,X)=0$ vanishes for all $i\leq n$. In order to keep our notation lighter, with a slight abuse we will sometimes write $\pi_i(Y,X)=0$, meaning $\pi_i (M_f,X)=0$ when the map $f$ is clear from the context.

We recall the following characterization of $n$-equivalence and the subsequent corollary.

\begin{pro}[\protect{\cite[Theorem 7.6.22]{Spanier}}]\label{p:Spanier7622}
If $f:X \rightarrow Y$ is an $n$-equivalence, then for every relative $CW$-complex $(K,L)$ with $K$ of dimension at most $n$, and every map $a:L\rightarrow X$ and $b:K\rightarrow Y$ with $b_{|L}=f \circ a$, there exists a map $c:K\rightarrow X$ with $c_{|L}=a$ and $f\circ c$ homotopic to $b$ relative to $L$.
\end{pro}

\begin{corollary}[\protect{\cite[Corollary 7.6.23]{Spanier}}]\label{p:Spanier7623}\label{c:Whitehead_Homotopyclasses}
If $K$ is a $CW$-complex and $f: X \rightarrow Y$ is an $n$-equivalence, then the induced homomorphism $$f_*:[K, X] \rightarrow[K, Y]$$ is a bijection if $\operatorname{dim} K<n$ and a surjection if $\operatorname{dim} K=n$. 
\end{corollary}

We recall a classical result due to Whitehead, which allows to deduce homotopic properties of a space from its cohomological ones, \textit{cfr.} also \cite[Theorem II.6]{Thom54}.

\begin{thm}\label{t:II.6}
    Let $f: X \rightarrow Y$ a map between two simply connected $CW$-complexes $X,Y$. If for any group coefficient $\mathbb{Z}_p$ the induced homomorphism $$f^*: \mathbf{H}^i(Y,\mathbb{Z}_p) \rightarrow \mathbf{H}^i(X,\mathbb{Z}_p)$$ is an isomorphism when $i<k$ and a monomorphism when $i=k$, then the relative homotopy groups $\pi_i(Y,X)=0$, for $i\leq k$.
\end{thm}

\begin{proof}
Consider the following exact sequence in cohomology:
$$\mathbf{H}^r(M_f) \xrightarrow{f^*} \mathbf{H}^r(X) \rightarrow \mathbf{H}^{r+1}(M_f,X) \rightarrow \mathbf{H}^{r+1}(M_f) \rightarrow \mathbf{H}^{r+1}(X).$$ The assumptions on $f$ are equivalent to $\mathbf{H}^{i}(M_f,X,\mathbb{Z}_p)=0$ for every prime $p$ and $i\leq k$. By duality on the group coefficients $\mathbb{Z}_p,$ we can write $\mathbf{H}_{i}(M_f,X,\mathbb{Z}_p)=0$ for $i\leq k$ which is equivalent, by the universal coefficient formula, to $\mathbf{H}_{i}(M_f,X,\mathbb{Z})=0$ for $i\leq k$. Since $X$ and $Y$ are simply connected, by the relative Hurewicz theorem, \textit{cfr.} \cite[Theorem 7.5.4]{Spanier}, we conclude that $\pi_i(M_f,X)=0$ for $i\leq k$, namely our claim (recall that by $\pi_i (Y,X)$ we actually mean $\pi_i (M_f, X)$).  
\end{proof}

 Let $\widetilde{G}_n(\mathbb{R}^{n+k})$ be the oriented Grassmannian manifold, that is the space of oriented $n$-dimensional subspaces in $\mathbb{R}^{n+k}$. The natural embedding $\mathbb{R}^n \hookrightarrow \mathbb{R}^{n+1}$, induces an embedding $\widetilde{G}_n(\mathbb{R}^{n}) \hookrightarrow \widetilde{G}_n(\mathbb{R}^{n+1})$. Thus, forming the union over increasing dimensions we obtain an infinite $CW$-complex named the \emph{infinite oriented Grassmannian} $$\widetilde{G}_n:=\widetilde{G}_n(\mathbb{R}^\infty)=\bigcup_k \widetilde{G}_n(\mathbb{R}^{n+k}),$$ as the set of all $n$-dimensional linear subspaces of $\mathbb{R}^{\infty}$, endowed with the direct limit topology. %: a set $U \subset \widetilde{G}_n$ is open if and only if all the sets $U \cap \widetilde{G}_n(\mathbb{R}^{n+k})$ for each $k$ are open.
We denote as $\widetilde{\gamma}^n$ the \emph{universal oriented n-plane bundle}, that is the canonical vector bundle over the base space $\widetilde{G}_n$ $$\widetilde{E} \xrightarrow{\pi} \widetilde{G}_n$$ with total space $\widetilde{E}$ consisting of pairs $(\ell, v) \in \widetilde{G}_n(\mathbb{R}^{\infty}) \times \mathbb{R}^{\infty}$ such that $v \in \ell$, topologized as a subset of the cartesian product, and with projection $\pi:\widetilde{E} \rightarrow \widetilde{G}_n$ such that $\pi(\ell,v)=\ell$.

Recall that any oriented $n$-plane\footnote{That is, all fibers are oriented $n$-dimensional real vector spaces.} bundle $\xi$ over a paracompact base $B$ admits a bundle map $\xi \rightarrow \widetilde{\gamma}^n$ and that any two bundle maps $f,g: \xi \rightarrow \widetilde{\gamma}^n$ from an oriented $n$-plane bundle $\xi$ to $\widetilde{\gamma}^n$ are bundle homotopic, meaning that there exists a one-parameter family of bundle maps $h_t : \xi \rightarrow \widetilde{\gamma}^n$, with $t \in [0,1]$ and $h_0=f$ and $h_1=g$ such that $h$ is continuous as a function of both variables, $\textit{cfr.}$ \cite[Theorems 5.6, 5.7]{MilnorStasheff}. %In other words, the associated function $h: E(\xi) \times [0,1] \rightarrow E(\widetilde{\gamma}_n)$ is continuous.
Hence, any oriented $n$-plane bundle $\xi$ over a paracompact space $B$ determines, up to orientation-preserving isomorphism, a unique homotopy class of maps $\overline{f}_{\xi}:B \rightarrow \widetilde{G}_n$. Since the classifying space $\widetilde{G}_n$ for oriented $n$-plane vector bundles is the classifying space associated to the rotations group $SO(n)$, we will denote it as usual by $\BSOn$. In fact, in almost all our considerations we just need to consider $\widetilde{G}_n (\mathbb R^{n+k})$ for a sufficiently large $k$ and at all effects treat $\BSOn$ as some fixed compact manifold $\widetilde{G}_n(\mathbb{R}^{n+k})$ for a suitably large $k$.

We recall now the main notions of Thom spaces and Thom's characterization of representability of a homology class. 

Let $\xi$ be an $n$-plane bundle with a Euclidean metric and $A\subset E(\xi)$ be the subset of the total space consisting of all vectors $v$ with $|v|\geq 1$. Then the identification space $E(\xi)/A$ is called the \emph{Thom space} $T(\xi)$ of $\xi$. Note that $T(\xi)$ has a preferred base point, denoted by $\infty$, and the complement $T(\xi) \setminus \{\infty\}$ consists of all vectors $v \in E(\xi)$ with $|v|<1.$ We note that if the base space $B$ of $\xi$ is a (finite) $CW$-complex, then the Thom space $T(\xi)$ is an $(n-1)$-connected (finite) $CW$-complex. 

If $\xi$ is a smooth oriented $n$-plane bundle, then the base space $B$ of $\xi$ can be smoothly embedded as the zero-cross section in the total space $E(\xi)$, and hence in the Thom space $T(\xi)$; moreover we note that the while $T(\xi)$ is not a manifold in general, the complement of the base point $T(\xi)\setminus \{\infty\}$ has the structure of a smooth manifold. 

Let $R$ be a commutative ring with unity and $\xi$ an $n$-plane bundle $E \xrightarrow{\pi} B$. For a point $b \in B$, let $S_b^n$ be the one-point compactification of the fiber $\pi^{-1}(b)$; since $S_b^n$ is the Thom space of $\xi_{|b}$, we have a canonical map 
\[
i_b: S_b^n \rightarrow T (\xi)\, .
\]
An \emph{$R$-orientation}, or a \emph{Thom class}, of $\xi$ is defined to be an element $u \in \widetilde{\mathbf{H}}^n(T (\xi) , R)$ (the reduced $n$-th cohomology group of $T (\xi)$) such that, for every point $b \in B$, $i_b^*(u)$ is a generator of (the free $R$-module) $\widetilde{\mathbf{H}}^n(S_b^n)$. We recall now a fundamental theorem, \textit{cfr.} \cite[Theorem 15.51]{Switzer}.

\begin{thm}[Thom isomorphism theorem]
    Let $u \in \widetilde{\mathbf{H}}^n(T (\xi), R)$ be a Thom class for an $n$-plane bundle $\xi$ of the form $E\xrightarrow{\pi} B$. Define $$ \Phi: \mathbf{H}^i(B , R) \rightarrow \widetilde{\mathbf{H}}^{n+i}(T (\xi), R)$$ by the cup product $\Phi(x)=\pi^*(x) \smile u$. Then $\Phi$ is an isomorphism for every integer $i$.
\end{thm}

We remark that for any oriented $n$-plane bundle the Thom class with $\mathbb{Z}$ coefficients exists and it is unique; analogously, for every $n$-plane bundle there exists a unique Thom class with $\mathbb{Z}_2$ coefficients, \textit{cfr.} \cite[Theorems 9.1, 8.1]{MilnorStasheff}.

%We refer to Appendix \ref{a:cohomology} for a brief discussion about cohomology operations and characteristic classes.

Thom's celebrated result about realizability of cycles by means of submanifolds can be stated as follows, \textit{cfr.} \cite[Théorème II.1]{Thom54}.

\begin{thm}\label{t:Thomclosed}
    Given $\mathcal{M}$ and $\tau$ as in Assumption \ref{a:1}, a homology class $\tau \in \mathbf{H}_{m}(\mathcal{M}, \mathbb{Z})$ is representable by a $m$-dimensional smooth submanifold $\Sigma \subset \mathcal{M}$ of codimension $n$ if and only if there exists a map $f: \mathcal{M} \rightarrow T(\widetilde{\gamma}^n)$ which pulls back the Thom class\footnote{Recall that, since $R= \mathbb Z$ and $n\geq 1$, the reduced cohomology group coincides with the standard cohomology group.} $u \in \mathbf{H}^n(T(\widetilde{\gamma}^n), \mathbb{Z})$ to the Poincaré dual of $\tau$. 
\end{thm}

%Thus, we will say that a cohomology class $x \in \mathbf{H}^n(\mathcal{M},\mathbb{Z})$ is \emph{representable by a smooth submanifold} if there exists a map $f: \mathcal{M} \rightarrow T(\widetilde{\gamma}^n)$ which pulls back the Thom class $u \in \mathbf{H}^n(T(\widetilde{\gamma}^n),\mathbb{Z})$ to $x$.

%We also recall another important result contained in Thom's seminal article: there is a bijection between cobordism classes of oriented submanifolds $\Sigma \subset \mathcal{M}$ and homotopy classes of maps $[\mathcal{M}, T(\widetilde{\gamma}^n)]$, \textit{cfr.} \cite[Théorème IV.6]{Thom54}.
By suitably modifying the proof of \cite[Théorème II.1]{Thom54}, one sees that the natural analog of Theorem \ref{t:Thomclosed} holds for compact manifolds with boundary. This is in fact what we will need in our arguments and we therefore provide a proof for the reader's convenience.

\begin{thm}\label{t:Thomboundary}
Let $\mathcal{M}$ be a connected smooth oriented compact $m+n$-dimensional Riemannian manifold with boundary $\partial \mathcal{M}$ and $\tau$ a nontrivial relative homology class $\tau \in \mathbf{H}_m(\mathcal{M},\partial \mathcal{M}, \mathbb{Z})$. Then $\tau$ is representable\footnote{With the obvious modifications in Definition \ref{d:smoothrepresentability} for $\mathcal{M},\Sigma$ with boundary and $\tau$ a relative homology class.} by a smooth compact embedded submanifold manifold $\Sigma \subset \mathcal{M}$ with $\partial \Sigma = \Sigma\cap \partial \mathcal{M}$ if and only if there exists a map $f: \mathcal{M} \rightarrow T(\widetilde{\gamma}^n)$ which pulls back the Thom class $u \in \mathbf{H}^n(T(\widetilde{\gamma}^n), \mathbb{Z})$ to the relative Poincaré dual of $\tau$. 
\end{thm}

\begin{proof} 
Assume that $\tau \in \mathbf{H}_m(\mathcal{M},\partial \mathcal{M}, \mathbb{Z})$ is representable by a smooth compact embedded submanifold $\Sigma \subset \mathcal{M}$ with $\partial \Sigma = \Sigma\cap \partial \mathcal{M}$, that is there exists a smooth embedding $h:\Sigma \rightarrow \mathcal{M}$ such that the fundamental class $[\Sigma]$ determined by the orientation of $\Sigma$ equals $\tau.$ Denote the relative Poincar\'e dual of $\tau$ by $x \in \mathbf{H}^{n}(\mathcal{M})$. If we consider $D(\mathcal{M})$ the double manifold of $\mathcal{M}$, then $D(\mathcal{M})$ is a smooth closed oriented manifold. In $D(\mathcal{M})$, we consider the double manifold of $\Sigma$, which is a smooth closed oriented embedded submanifold of $D(\mathcal{M})$, whose fundamental class is an absolute homology class in $\mathbf{H}_m(D(\mathcal{M}))$; denote by $y\in \mathbf{H}^{n}(D(\mathcal{M}))$ its Poincaré dual. By applying Thom's construction of Theorem \ref{t:Thomclosed}, we find a map $$F:D(\mathcal{M}) \rightarrow T(\univob)$$ such that $F^{*}(u)=y$. Denoting by $i: \mathcal{M} \rightarrow D(\mathcal{M})$ the inclusion map, we consider the restriction of $F$ to $\mathcal{M}$, so that we obtain a new map $f:\mathcal{M} \rightarrow T(\univob)$ such that $$f^*(u)= (F\circ i)^*(u)=i^*(y)=x.$$

Conversely, assume that there exists a map $f:\mathcal{M}\rightarrow T(\univob)$ such that $f^*(u)=x$. Consider the restriction of $f$ to $\partial \mathcal{M}$ and denote it by $\partial f$. The space $T(\univob)\setminus \{\infty\}$ is a smooth manifold: hence by \cite[Proposition 2.3.4]{Wall}, we can approximate the map of $\partial f$ by means of a new map $g_0$ agreeing with $\partial f$ on $\partial f^{-1}(U(\infty))$ and which is of class $C^{\infty}$ on $\partial\mathcal{M}\setminus \partial f^{-1}(U(\infty))$, where $U$ is a small smooth neighborhood of $\{\infty\}$; if the approximation is close enough, then $g_0$ is homotopic to $\partial f$. By the standard density argument, \textit{cfr.} \cite[Theorem 4.5.6]{Wall}, we can approximate $g_0$ by a homotopic map $g_1$ which is smooth and transverse to the zero section\footnote{With the usual approximation to the restriction of $\univob$ to a sufficiently large compact manifold $\widetilde{G}_n(\mathbb{R}^{n+k})$.} of $\univob$. Since $\partial \mathcal{M}$ has a collar neighborhood, by the homotopy extension property there is a map $f_1$ defined on $\mathcal{M}$, homotopic to $f$ and such that $\partial f_1=g_1$. By \cite[Proposition 2.3.4 (ii)]{Wall}, we can assume that $f_1$ is smooth and coincides with $g_1$ on $\partial \mathcal{M}$. By \cite[Proposition 4.5.7]{Wall}, we obtain a final map $\widetilde{f}$ arbitrarily close to $f$, agreeing with it on $f^{-1}(U(\infty))$, of class $C^{\infty}$ on $\mathcal{M}\setminus f^{-1}(U(\infty))$ and such that both $\widetilde{f}$ and $\partial \widetilde{f}$ are transverse to the zero section, \textit{cfr.} also \cite[Proposition 4.5.10]{Wall}. By applying the same proof as in \cite[Theorem 8.2]{Benedetti}, the preimage $\widetilde{f}^{-1}(\BSOn)$ is a smooth compact embedded submanifold $\Sigma \subset \mathcal{M}$ of codimension $n$, with $\partial \Sigma = \Sigma\cap \partial \mathcal{M}$. Hence we can conclude that $x =f^*(u)=\widetilde{f}^*(u)$.
\end{proof}

\begin{remark}
    A direct way to prove the first implication of Theorem \ref{t:Thomboundary} has been suggested to us by Jacob Lurie and it is the following. Let $\mathcal{M}, \tau, \Sigma$ as above. Choose an extension of the normal bundle $\nu_{\Sigma}$ of $\Sigma$ to some smooth $n$-plane bundle $\xi:E \rightarrow U$, where $U$ is an open neighborhood of $\Sigma$: this induces an isomorphism $q: \nu_{\Sigma} \rightarrow \xi|_{\Sigma}$; choose a connection on $\xi$. Working locally in each chart and then using a partition of unity, it is possible to construct a smooth section $s$ of $\xi$ which vanishes on $\Sigma$ and such that the covariant derivative of $s$ along $\Sigma$ induces the isomorphism $q$. Up to shrinking the open set $U$, one can assume that $s$ vanishes exactly on $\Sigma$ and that it is transverse to the zero section. Now it is possible to conclude in analogy with Thom's construction: choose a bundle map $\xi \rightarrow \widetilde{\gamma}^n$ and, after having rescaled $s$, note that it is possible to extend the classifying map with domain $U$ and with values in the corresponding subspace of the total space of $\widetilde{\gamma}^n$ to a map $f:\mathcal{M}\rightarrow T(\univob)$, sending the complement of $U$ to the 0-cell of $T(\univob)$. It follows that $x=f^*(u)$, where $x$ is the Poincaré dual of $\tau$ and $u\in \mathbf{H}^n(T(\univob),\mathbb{Z})$ the Thom's class. This argument shows that Thom's construction can be performed without any need of a tubular neighborhood theorem. 
    \end{remark}

We refer to Appendix \ref{a:cohomology} for a brief discussion about cohomology operations and characteristic classes, needed in the proof of Lemma \ref{l:m+4_equivalenza}.

\section{Triangulations, skeleta, neighborhoods, and maps}\label{s:smoothingneighborhood}

\subsection{Suitable neighborhoods of skeleta}\label{s:subsmoothingneighborhood}
Having fixed a triangulation $\mathcal{K}$ of $\mathcal{M}$ and a skeleton $\mathcal{K}^j$, we denote by $B_\delta (\mathcal{K}^j)$ the usual metric neighborhoods of the skeleton, namely the sets of points $p$ with $\dist (p, \mathcal{K}^j) < \delta$. In many instances we will use these neighborhoods for our considerations. However, for some important considerations we will in fact need a suitable variant, which will be denoted by
$V_\delta (\mathcal{K}^j)$ and are defined in the following way. We first fix a (sufficiently large) constant $C_0$ which will depend on the triangulation $\mathcal{K}$, subdivide the simplices forming $\mathcal{K}^j$ into $\mathcal{S}_0\cup \ldots \cup \mathcal{S}_j$ according to their dimension ($\mathcal{S}_i$ being the collection of simplices of dimension $i$) and hence
set
\begin{equation}\label{e:intorni-spigolosi}
V_\delta (\mathcal{K}^j) := \bigcup_{i=0}^j \bigcup_{\sigma \in \mathcal{S}_i} B_{C_0^{-i} \delta} (\sigma)\, 
\end{equation}
where 
\[
B_{C_0^{-i} \delta} (\sigma) = \{p: \dist (p, \sigma) < C_0^{-i} \delta\}\, .
\]
The following is an elementary consequence of our definition.

\begin{lemma}\label{l:spigoli}
For every triangulation $\mathcal{K}$ of $\mathcal{M}$ and every $j\leq m+n-1$ there is a choice of $\bar \delta>0$ (sufficiently small) and of $C_0$ sufficiently large such that the following holds. 
First of all, $\mathcal{M}\setminus V_{\bar\delta} (\mathcal{K}^j)$ is a deformation retract of $\mathcal{M}\setminus \mathcal{K}^j$. 

Moreover, for every point $p\in \partial V_{\bar\delta} (\mathcal{K}^j)$ there is at most one $\sigma$ in each $\mathcal{S}_i$ (with $0\leq i \leq j$) such that $p\in \partial B_{C_0^{-i} \bar\delta} (\sigma)$. In particular, there is a neighborhood of $U$ of $p$, an integer $\bar j\in \{1, \ldots, j\}$ and a diffeomorphism $\phi: U \to B_1 \subset \mathbb R^{m+n}$ (the unit ball in $\mathbb R^{m+n}$) such that 
\[
\phi (U\setminus V_{\bar\delta} (\mathcal{K}^j)) = \{(x_1, \ldots , x_{m+n}) : x_i>0 \mbox{ for $1\leq i \leq \bar{j}$}\}\, .
\]
\end{lemma}

Note in particular that the boundary of $V_{\bar\delta} (\mathcal{K}^j)$ is a Lipschitz submanifold. This is, however, not suitable for our purposes; we need an appropriate regularization of it which, given the explicit local description of Lemma \ref{l:spigoli}, is a consequence of a standard regularization procedure. 

\begin{lemma}\label{l:spigoli-allisciati}
Let $\mathcal{K}$ be a triangulation of $\mathcal{M}$, let $\bar \delta$ and $C_0$ be given by Lemma \ref{l:spigoli}, and fix any pair of positive numbers $\delta' <\delta<\bar\delta$. Then there is a neighborhood $U_\delta (\mathcal{K}^j)$ of $\mathcal{K}^j$ with the following properties:
\begin{itemize}
\item $V_\delta (\mathcal{K}^j) \supset U_\delta (\mathcal{K}^j) \supset V_{\delta'} (\mathcal{K}^j)$;
\item The boundary of $U_\delta (\mathcal{K}^j)$ is smooth;
\item $\mathcal{M}\setminus U_\delta (\mathcal{K}^j)$ is a deformation retract of $\mathcal{M}\setminus V_{\delta'} (\mathcal{K}^j)$;
\item There is a smooth tubular neighborhood $\mathcal{C}$ of $\partial U_\delta (\mathcal{K}^j)$ in $\mathcal{M}$ containing $\partial V_{\delta'} (\mathcal{K}^j)$. 
\end{itemize}
\end{lemma}

\subsection{Maps} We will now build some special maps related to the neighborhoods $B_\delta$ and $V_\delta$.

\begin{lem}\label{l:Phi}
Let $\mathcal{M}$ be as in Assumption \ref{a:1} and $\mathcal{K}$ a triangulation of $\mathcal{M}$. For every $\varepsilon_a>0$ and $ \eta_a >0$ there is a positive number $\delta_a <  \eta_a$ with the following property. If $\gamma\in ]0,1]$ is an arbitrary number, then there is a diffeomorphism $\Phi$ such that:
\begin{enumerate}\itemsep0.2em
\item $\Phi$ is isotopic to the identity and it coincides with the identity on $\mathcal{M}\setminus B_{\eta_a} (\mathcal{K}^k)$;
\item $\Lip (\Phi) \leq 1+ \varepsilon_a$;
\item For every point $p\in B_{\delta_a} (\mathcal{K}^k)$ there is an orthonormal frame $e_1, \ldots, e_{m+n}$ such that 
\begin{align}
&|d\Phi_p (e_i)|\leq 1+ \varepsilon_a \qquad &\forall i\in \{1, \ldots, k\},\\
&|d\Phi_p (e_j)|\leq \gamma &\forall j\in \{k+1, \ldots, m+n\}\, .
\end{align}
\end{enumerate}
\end{lem}

Before coming to the proof of the latter lemma, we remark that a simple modification of the arguments gives the following one, which is in fact much simpler.

\begin{lemma}\label{l:second-Phi}
Let $\mathcal{M}$ be as in Assumption \ref{a:1}, $\mathcal{K}$ a triangulation, $j\in \{0, \ldots, m+n-1\}$. If $C_0$ and $\bar\delta^{-1}$ in Lemma \ref{l:spigoli} are chosen sufficiently large, then for every $\delta_b> \delta'_b>0$ there is a Lipschitz map $\Phi: \mathcal{M}\to \mathcal{M}$ with the following properties:
\begin{itemize}
\item $\Phi$ maps $V_{\delta'_b} (\mathcal{K}^j)$ into $\mathcal{K}^j$;
\item $\Phi$ is a smooth diffeomorphism between $\mathcal{M}\setminus \Phi^{-1} (\mathcal{K}^j)$ and $\mathcal{M}\setminus \mathcal{K}^j$;
\item $\Phi (p)=p$ for every $p\not\in V_{\delta_b} (\mathcal{K}^j)$. 
\end{itemize}
\end{lemma}

\begin{proof}[Proof of Lemma \ref{l:Phi}] The proof is by induction over $k$.

\medskip

We start with the first step, where $k=0$. We enumerate the $0$-skeleton as the points $p_1, \ldots, p_N$, we let $d$ be the minimum of $\dist (p_i, p_j)$ and $r_0$ a radius which is smaller than the minimum of the injectivity radii for the exponential maps centered at $p_i$ and whose choice will be specified later. We then set $\eta:= \min \{\eta_a, \frac{d}{4}, \frac{r_0}{2}\}$. We fix $\delta_a < \eta$ and $\mu\geq 0$ (whose choice will be specified later) and let $\varphi: [0, \infty[ \to [0, \infty[$ be the following piecewise linear increasing function:
\[
\varphi (t) = \left\{
\begin{array}{ll}
\mu t \qquad &\mbox{if $0\leq t \leq 2\delta_a$,}\\
\frac{\eta- \delta_a - 2\mu \delta_a}{\eta-3\delta_a} (t-2\delta_a) + 2\mu \delta_a \qquad &\mbox{if $2\delta_a \leq t \leq \eta-\delta_a$,}\\
t &\mbox{if $t\geq \eta-\delta_a$.}
\end{array}
\right.
\]
Observe that $0\leq \varphi'\equiv \mu$ on $[0, 2\delta_a]$ while $0\leq \varphi' \leq \frac{\eta}{\eta-3\delta_a}$ everywhere else, and the latter number can be made arbitrarily close to 1 depending only on the ratio $\frac{\delta_a}{\eta}$, but independently of $\mu$. We then regularize $\varphi$ by convolution with a standard smooth nonnegative kernel, hence getting a smooth diffeomorphism of the real half line, which we denote by $\psi$. Its derivative $\psi'$ will enjoy the same global upper bound and, by choosing the kernel suitably, we can ensure $\psi (t)=\mu t$ on the interval $[0, \delta_a]$ and $\psi (t)=t$ on $[\eta-\frac{\delta_a}{2}, \infty[$. We next define the map $\Psi_\mu (x) := \psi (|x|) \frac{x}{|x|}$ from $\mathbb R^{m+n}$ onto itself. Notice that $\Lip (\Psi_\mu)$ can be made arbitrarily close to $1$ choosing the ratio $\frac{\delta_a}{\eta}$ very small, while clearly $\Psi_\mu (x)=\mu x$ in the ball of radius $B_{\delta_a}$. We are now ready to define the map $\Phi$. $\Phi (p):=p$ for $p\not \in \mathcal{M}\setminus B_{\eta} (\mathcal{K}^0)$. Next $B_\eta (\mathcal{K}^0)$ is the disjoint union of $B_{\eta} (p_i)$. On each such ball we consider the exponential map ${\rm exp}_{p_i}$ and $\Phi$ is defined to be 
\[
\Phi := {\rm exp}_{p_i} \circ \Psi_\mu \circ {\rm exp}_{p_i}^{-1}\, . 
\]
By choosing the radius $r_0$ sufficiently small we can get the Lipschitz constant of the exponential maps and of their inverses arbitrarily close to $1$ on the domains of our interest. So, if we fix some constant $\tilde{\varepsilon}$, after choosing $\frac{\delta_a}{\eta}$ and $r_0$ sufficiently small, we can easily achieve
\[
{\rm Lip}\, (\Phi) \leq (1+\tilde{\varepsilon})^3
\]
and 
\[
{\rm Lip}\, (\Phi_{|B_{\delta_a} (\mathcal{K}^0)}) \leq (1+\tilde{\varepsilon})^2 \mu
\]
In particular we first choose $\tilde{\varepsilon}$ so that $(1+\tilde{\varepsilon})^3 \leq 1+ \varepsilon_a$ and we then choose $\mu$ so that $(1+\tilde{\varepsilon})^2 \mu\leq \gamma$. Note that the choice of $\delta_a$ is then independent of $\gamma$.

\medskip

We next wish to tackle the induction step, so we assume that the statement of the proposition holds for $k-1$ in place of $k$. We then fix $\varepsilon_a$, $\eta_a$, and $\gamma$. We apply the proposition in case $k-1$ with the same $\eta_a$ but with $\gamma_0$ and $\varepsilon_0$ in place of $\gamma$ and $\varepsilon_a$ and get the corresponding $\delta_a$, which we denote by $\delta_0$. The choices of $\varepsilon_0$ and $\gamma_0$ will be specified later, but we anticipate that $\varepsilon_0$ will only depend on $\varepsilon_a$ among all these parameters. We therefore then have a corresponding map, which we denote by $\Phi_{k-1}$, with the property that ${\rm Lip} (\Phi_{k-1})\leq 1+ \varepsilon_0$, which is the identity outside of $B_{\eta_a} (\mathcal{K}^{k-1})$ and which in turn satisfies all the requirements of the lemma for the other parameters. 

Next we consider $\delta_a, \eta$, and $\mu$, whose choices will be specified later. For each $k$-dimensional face $F$ in the $k$-dimensional skeleton, we consider 
\[
F':= F\setminus B_{\delta_0/16} (\mathcal{K}^{k-1}) 
\]
and we choose $\eta$ small enough so that the normal neighborhoods $N_\eta (F')$ are pairwise disjoint and all diffeomorphic to $F'\times B^{m+n-k}_\eta$, where $B^{m+n-k}_\eta$ denotes the $m+n-k$ dimensional ball of radius $\eta$ and centered at $0$ in $\mathbb R^{m+n-k}$. We then parametrize $N_\eta (F')$ as $(x,y)$, where $x\in F'$ and $y\in B^{m+n-k}_\eta$. 

We next introduce a function of two variables defined in the following way. First of all we define
$\bar \mu: [0, \infty[ \to [0, \infty[$ as
\[
\bar \mu (s) = \left\{\begin{array}{ll}
1 \qquad &\mbox{if $s\leq \frac{\delta_0}{4}$},\\
1- 2 \frac{(1-\mu)}{\delta_0} (s-{\textstyle{\frac{\delta_0}{4}}}) \qquad &\mbox{if $\frac{\delta_0}{4}\leq s \leq \frac{3\delta_0}{4}$},\\
\mu \qquad &\mbox{if $s\geq \frac{3\delta_0}{4}$.}
\end{array}
\right.
\]
Hence we set
\[
\varphi (s,t) = \left\{
\begin{array}{ll}
t \bar \mu (s) \qquad &\mbox{if $t\leq 2\delta_a$},\\
\frac{\eta-\delta_a -2\bar \mu (s) \delta_a}{\eta-3\delta_a} (t-2\delta_a) + 2\bar \mu (s) \delta_a \qquad &\mbox{if $2\delta_a \leq t \leq \eta-\delta_a$},\\
t &\mbox{if $t\geq \eta-\delta_a$}.
\end{array}
\right.
\]
Note that $\partial_t{\varphi} \equiv \mu$ if $t\leq 2\delta_a$ and $s\geq \frac{3\delta_0}{4}$ while $\partial_t{\varphi} \equiv 1$ if $t\geq \eta-\delta_a$ or for every $t$ if $s\leq \frac{\delta_0}{4}$. On the other hand we have the upper bound
\[
\left|\frac{\partial \varphi}{\partial t}\right|\leq \frac{\eta}{\eta-3\delta_a}\, ,
\]
where the right hand side can be made arbitrarily close to $1$ by choosing $\frac{\delta_a}{\eta}$ small. Likewise we have the upper bound
\[
\left|\frac{\partial \varphi}{\partial s}\right| \leq \frac{4\eta}{\delta_0}
\]
and the right hand side can be made arbitrarily small by choosing $\frac{\eta}{\delta_0}$ small. 

These two requirements (namely $\frac{\eta}{\eta-3\delta_a}$ being sufficiently close to $1$ and $\frac{4\eta}{\delta_0}$ being sufficiently close to $0$) will only depend on $\varepsilon_a$, $\varepsilon_0$ and the geometry of the triangulation, while $\varepsilon_0$ will only depend on $\varepsilon_a$ and the geometry of the triangulation, so that ultimately $\delta_a$ will depend in fact only on $\varepsilon_a$, $\eta_a$, $\mathcal{M}$ and $\mathcal{K}$. 

With a similar regularization procedure as the one outlined above, we can smooth $\varphi$ to a function $\psi$. Then we also suitably smooth the distance function $p\mapsto \dist (p, \mathcal{K}^{k-1})$ in $B_{\delta_0} (\mathcal{K}^{k-1})\setminus B_{\frac{\delta_0}{8}} (\mathcal{K}^{k-1})$ to a function $d$. We then define a function $\Phi_k : N_\eta (F')\to N_\eta (F')$ by setting 
\[
\Phi_k (x,y) := \left(x,  \psi (d(x,0), |y|){\textstyle{\frac{y}{|y|}}}\right)\, .
\]
Being the $N_\eta (F')$ pairwise disjoint, this define a function on the union of them. Finally, since the function is the identity at the boundary of this domain, we can extend it to all of $\mathcal{M}$ by being the identity. We then claim that the function $\Phi:= \Phi_{k-1}\circ \Phi_k$ in fact satisfies all the requirements upon choosing our parameters correctly. 
The bound on the Lipschitz constant simply follows by multiplying the bounds of the Lipschitz constants of the two maps and choosing the parameters correctly. The fact that the map is equal to the identity outside of $B_{\eta_a} (\mathcal{K}^k)$ follows from choosing $\eta\leq \eta_a$. It remains to check the third claim. We will check that the claim holds at every $p$ such that $\dist (p, \mathcal{K}^k)\leq \delta_a$ but $\dist (p, \mathcal{K}^{k-1})\geq 3\frac{\delta_0}{4}$ and at every $p$ such that $\dist (p, \mathcal{K}^{k-1})\leq 3\frac{\delta_0}{4}$. The union of the two sets clearly contains $B_{\delta_a} (\mathcal{K}^k)$, and this completes the proof. First of all observe that if $p$ is in the first set, then by construction there are $m+n-k$ orthonormal vectors $e_{k+1}, \ldots, e_{m+n}$ with the property that 
\[
|d \Phi_{k}|_p (e_i)|\leq 2\mu\, .
\]
On the other hand because of the Lipschitz bound on $\Phi_{k-1}$ we immediately conclude that 
\[
|d\Phi|_p(e_i)|\leq 2 (1+\varepsilon_0)\mu
\]
and thus choosing $\mu$ appropriately we can guarantee $2 (1+\varepsilon_0) \mu \leq \gamma$. Completing the $e_i$'s to an orthonormal basis we get the desired estimate on the other vectors simply using the global Lipschitz estimate on $\Phi$.

Consider now a $p\in B_{3\delta_0/4} (\mathcal{K}^{k-1})$. By construction $q=\Phi_{k-1} (p)$ belongs to $B_{\delta_0} (\mathcal{K}^{k-1})$. There are therefore $m+n+1-k$ orthonormal vectors $v_k, \ldots , v_{m+n}$ with the property that 
\[
|d\Phi_{k-1}|_q (v_i)|\leq \gamma_0\, . 
\]
Consider the vector space $V$ spanned by these vectors. Then we have the estimate
\[
|d\Phi_{k-1} (v)|\leq \sqrt{m+n+1-k}\, \gamma_0 |v|
\]
for every such $v$. Because $\Phi$ is a diffeomorphism, there is an $m+n-k$-dimensional subspace $W$ of $T_p \mathcal{M}$ which $d\Phi|_p$ maps onto $V$. If we choose an orthonormal base $e_{k+1}, \ldots , e_{m+n}$ of the latter, we can then estimate
\[
|d\Phi_p (e_i)|\leq \sqrt{m+n-k}\, \gamma_0\, \Lip (\Phi_k) \, .
\]
We then conclude by choosing $\gamma_0 \leq \frac{\gamma}{2\sqrt{m+n-k}}$, given that our constructions certainly implies $\Lip (\Phi_k)\leq 1+\varepsilon_a\leq 2$.
\end{proof}

\section{Tools from geometric measure theory}\label{s:GMT}

The proof of our main theorems will make use of two technical propositions from geometric measure theory, which do not involve the knowledge of sophisticated topological tools.

\begin{pro} \label{p:poly_approx_prescribedsing}
Let $\mathcal{M}$ be as in Assumption \ref{a:1} and $T$ be an integral $m$-dimensional cycle in $\mathcal{M}$. 
For every fixed $\varepsilon_c > 0$ there is an integral cycle $P$ homologous to $T$ and a smooth triangulation $\mathcal{K}$ of $\mathcal{M}$ with the following properties:
\begin{itemize}
\item[($a_0$)] $\mathbb{M}(P) \leq (1+\varepsilon_c) \mathbb{M}(T)$;
\item[($b_0$)] $\mathbb{F} (T-P)\leq \varepsilon_c$;
\item[($c_0$)] $\spt(P) \subset \{ x : \dist(x,\spt(T))\leq \varepsilon_c\}$;
\item[($d_0$)] $P = \sum_{F\in \mathcal{F}^m} \beta_F \llbracket F \rrbracket$, where $\beta_F\in \mathbb Z$ and $\mathcal{F}^m$ is the collection of $m$-dimensional cells of $\mathcal{K}$ with an appropriately chosen orientation.
\end{itemize}
Furthermore, for a sufficiently small $\delta_c'>0$ and any $\delta_c<\delta_c'$ we can find a second integral cycle $P'$ homologous to $P$ with the following properties:
\begin{itemize}
\item[(a)] $\mathbb{M} (P') \leq (1+3\varepsilon_c) \mathbb{M} (T)$ and $\mathbb{F} (T-P') \leq 3\varepsilon_c$;
\item[(b)] $\spt(P') \subset \{ x : \dist(x,\spt(T))\leq 3\varepsilon_c\}$;
\item[(c)] $\|P'\| (B_{\delta'_c} (\mathcal{K}^{m-2})) \leq 3\varepsilon_c$;
\item[(d)] $P'\res \mathcal{M}\setminus B_{\delta_c} (\mathcal{K}^{m-2})= \llbracket\Gamma \rrbracket$ for some smooth oriented submanifold $\Gamma$ of $\mathcal{M}\setminus B_{\delta_c} (\mathcal{K}^{m-2})$ without boundary in $\mathcal{M}\setminus B_{\delta_c} (\mathcal{K}^{m-2})$.
\end{itemize}
\end{pro}

\begin{remark}\label{r:m-2corepresentability}
A routine modification of the arguments used to prove Proposition \ref{p:poly_approx_prescribedsing} implies in fact that the cycle $P'$ can be chosen so that its singularities are all {\em contained} in $\mathcal{K}^{m-2}$. While this is a much weaker result than the one achieved by our main theorem of constructing an integral cycle with singularities all contained in $\mathcal{K}^{m-5}$, its proof can however be completed without recurring to any sophisticated topological fact.  
\end{remark}

In the second proposition we are given two $m$-dimensional integral cycles $S$ and $R$ which agree outside of a sufficiently small neighborhood of the $m-2$-dimensional skeleton $\mathcal{K}^{m-2}$. We will then show that:
\begin{itemize}
\item $S$ and $R$ represent the same homology class;
\item There is a smooth deformation $R'$ of $R$ which is close, in terms of mass and in flat norm, to $S$;
\item $R'$ coincides with $R$ outside a slightly larger neighborhood of $\mathcal{K}^{m-2}$.
\end{itemize}

\begin{pro}\label{p:squash}
Let $m$ and $\mathcal{M}$ be as in Assumption \ref{a:1} and let $\mathcal{K}$ be a smooth triangulation of $\mathcal{M}$. Then for every $\varepsilon_d >0$ and every $\eta_d>0$ there exists $\delta_d(\varepsilon_d, \eta_d, \mathcal{K}, \mathcal{M})>0$ with the following property. Suppose $S$ and $R$ are $m$-dimensional integral cycles in $\mathcal{M}$ and that 
\[
S\res\mathcal{M}\setminus B_{\delta_d}\left(\mathcal{K}^{m-2}\right)=R\res\mathcal{M}\setminus B_{\delta_d}\left(\mathcal{K}^{m-2}\right)\, .
\]
Then $S$ and $R$ are homologous and moreover there 
exist an integral cycle $R'$ in their homology class and a diffeomorphism $\Phi$ of $\mathcal{M}$ with the following properties:
\begin{enumerate}\itemsep0.2em
\item $\mathbb{M}(R') \leq (1+\varepsilon_d) \mathbb{M}(S)$;
\item $\mathbb{F}(R'-S)\leq C (\varepsilon_d \mathbb{M} (S)+2\|S\| (B_{\eta_d} (\mathcal{K}^{m-2})))^{\frac{m+1}{m}}$, with $C=C (\mathcal{M})$;
\item $R'\res\mathcal{M}\setminus B_{\eta_d}\left(\mathcal{K}^{m-2}\right)=S\res\mathcal{M}\setminus B_{\eta_d}\left(\mathcal{K}^{m-2}\right)$;
\item $\Phi$ is in the isotopy class of the identity and $R'= \Phi_\sharp R$.
\end{enumerate}
\end{pro}

\begin{remark}
Note that the parameter $\delta_d$ does not depend on $S$ and $R$.
Its dependence on the parameters $\varepsilon_d$ and $\eta_d$ can be computed through our arguments, but since such explicit dependence is irrelevant for our purposes, we will ignore the issue.
\end{remark}

\subsection{Proof of Proposition \ref{p:poly_approx_prescribedsing}: first approximation}

In this section we show the existence of the first approximating cycle $P$ as in Proposition \ref{p:poly_approx_prescribedsing}. It is quite possible that the existence of a $P$ with the desired properties is already proved in the existing literature; nevertheless, we have not been able to find a precise reference for our purposes and therefore we provide a proof. Instrumental to our argument is to consider the ambient manifold $\mathcal{M}$ to be smoothly isometrically embedded in some Euclidean space $\mathbb R^N$ (the codimension is irrelevant), which we can always assume without loss of generality thanks to Nash's Theorem.  

\medskip

%Following \cite[4.1.22]{Federerbook}, recall that we denote by $\mathscr{P}_m (\mathbb R^N)$ the set of integral polyhedral chains on $\mathbb R^N$, namely the additive group generated by $m$-dimensional simplices of $\mathbb R^N$. Clearly the latter is a subset of the space of integral $m$-dimensional currents. 
Consider now the integral cycle $T$ in $\mathcal{M}$ as an integral cycle of $\mathbb R^N$. By \cite[Lemma 4.2.19]{Federerbook} for every $\kappa>0$ there is a diffeomorphism $g$ and an integral $m+1$-dimensional current $S$ such that 
\begin{itemize}
\item ${\rm Lip}\, (g) \leq 1 +\kappa$ and $|g(x)-x|\leq \kappa$ for all $x$;
\item $\mathbb{M} (S) + \mathbb{M} (\partial S)\leq \kappa$;
\item ${\rm spt}\, (S) \subset \{x: \dist (x, {\rm spt}\, (T)) \leq \kappa\}$;
\item $g_\sharp T + \partial S \in \mathscr{P}_m (\mathbb R^N)$.
\end{itemize}
Note therefore that, if we set $\bar{P}:= g_\sharp T + \partial S$, then $\bar{P}$ is a cycle,
\[
\mathbb{M} (\bar{P}) \leq (1+\kappa) \mathbb{M} (T) + \kappa
\]
and 
\[
{\rm spt}\, (\bar P) \subset \{x: \dist (x, {\rm spt}\, (T)) \leq \kappa\} \subset B_\kappa(\mathcal{M})\, .
\]
We next observe that we can, without loss of generality, regard $\bar{P}$ as 
\[
\bar{P} = \sum_i k_i \llbracket P_i\rrbracket
\]
where each $k_i$ is a positive integer, each $P_i$ is an oriented {\em closed} simplex, and for every pair of distinct $P_i$ and $P_j$, either $P_i\cap P_j=\emptyset$ or $P_i\cap P_j$ is a common lower-dimensional face. This can be seen as follows: fix a representation as $\bar{P} = \sum_j \bar k_j \llbracket \bar P_j\rrbracket$ with $\bar k_j$ positive integers and $\bar P_j$ oriented simplices. Fix a triangulation $\mathcal{T}$ of $\mathbb R^N$ and apply Proposition \ref{p:nello-scheletro} to refine $\mathcal{T}$ to a new triangulation $\mathcal{T}'$ with the property that each $\bar P_j$ is the union of $m$-dimensional simplices in $\mathcal{T}'$ (elements of the $m$-skeleton). The desired conclusion is then immediate. 

Assuming $\kappa$ to be sufficiently small, we can further assume that the orthogonal projection $\mathbf{p}: B_\kappa (\mathcal{M})\to \mathcal{M}$ is smooth, well-defined and with ${\rm Lip}\, (\mathbf{p})$ arbitrarily close to $1$. More precisely, ${\rm Lip}\, (\mathbf{p}) = 1+\bar{\kappa}$ where $\bar{\kappa}\downarrow 0$ as $\kappa\downarrow 0$. 

\medskip

We now consider the cycle $\mathbf{p}_\sharp \bar{P}$. Because of the usual homotopy formula, and because of the estimate above, we can ensure that 
\begin{align*}
\mathbb{M} (\mathbf{p}_\sharp \bar{P}) &\leq (1+ C \bar \kappa) \mathbb{M} (\bar P) \leq (1+C\bar\kappa) (\mathbb{M} (T) + \kappa),\\
\mathbb{F} (T-\mathbf{p}_\sharp \bar{P}) &\leq \mathbb{F} (T-\bar P) + \mathbb{F} (\bar P-\mathbf{p}_\sharp \bar{P}) \leq \kappa + C \bar\kappa \mathbb{M} (\bar P)\\
&\leq \kappa + C \bar \kappa (\mathbb{M} (T) + \kappa),\\
{\rm spt}\, (\mathbf{p}_\sharp \bar{P}) &\subset B_\kappa ({\rm spt} (\bar P)) \subset B_{2\kappa} ({\rm spt}\, (T))\, .
\end{align*}

\medskip

Consider now that that $K:=\spt (\bar P)$ is a polyhedron in the sense of Definition \ref{d:polyhedra} and that $\mathbf{p}: K \to \mathcal{M}$ is a piecewise smooth map in the sense of Definition \ref{d:PL-manifolds}. We next fix a triangulation $\mathcal{T}$ of $\mathcal{M}$, which again we understand as a piecewise smooth homeomorphism $\varphi$ of some PL-submanifold $L\subset \mathbb R^{\bar N}$ onto $\mathcal{M}$, according to Definition \ref{d:PL-manifolds}. We next recall the following proposition about uniqueness of smooth triangulations, which is due to Whitehead, \textit{cfr.} \cite{Whitehead}, and corresponding to \cite[Lect. 5, Theorem 1]{Jacob}:

\begin{pro}\label{p:whitehead}
Consider two piecewise smooth homeomorphisms $f: L \to \mathcal{M}$ and $g: M \to \mathcal{M}$ where $L\subset \mathbb R^{N_1}$ and $M\subset \mathbb R^{N_2}$ are two finite polyhedra. Then for every $\eta>0$ there are two piecewise smooth homeomorphisms $f': L \to \mathcal{M}$ and $g': M\to \mathcal{M}$ which are $\eta$-close in the $C^1$-sense to $f$ and $g$ and such that $f'^{-1}\circ g'$ and $g'^{-1} \circ f'$ are piecewise linear. 
\end{pro}

Recall that being $f$ and $f'$ piecewise smooth, for both there are triangulations $\mathcal{T}$ and $\mathcal{T}'$ of $L$ with the property that the restriction of each of the simplices of the corresponding triangulation 
is a smooth function. Closeness in the $C^1$-sense means that $\|f_{|{\Delta\cap \Delta'}} - f'_{|{\Delta \cap \Delta'}}\|_{C^1} \leq \eta$ for every $\Delta \in \mathcal{T}$ and $\Delta'\in \mathcal{T}'$.

An inspection of the argument given in \cite{Jacob} shows that the invertibility of both maps is only used to prove the piecewise linearity of both $f'^{-1}\circ g'$ and $g'^{-1}\circ f'$. Adapted to our setting, the arguments lead to the following conclusion.

\begin{pro}\label{p:Whitehead}
For every $\eta>0$ there is a piecewise smooth homeomorphism $\psi: L \to \mathcal{M}$ and a piecewise smooth map $\mathbf{q}: K \to \mathcal{M}$ such that:
\begin{itemize}
\item[($\alpha$)] $\psi$ and $\mathbf{q}$ are $\eta$-close in the $C^1$-sense to $\varphi$ and $\mathbf{p}$;
\item[($\beta$)] $\Psi := \psi^{-1}\circ \mathbf{q}: K \to L$ is piecewise linear.
\end{itemize}
\end{pro}

Point ($\beta$) means that there is a triangulation $\mathcal{T}_1$ of $K$ and a triangulation $\mathcal{T}_2$ of $L$ with the property that every simplex $\Delta$ in the triangulation $\mathcal{T}_1$ is mapped by $\psi^{-1} \circ \mathbf{q}$ inside some simplex $\Delta'$ of $\mathcal{T}_2$ and that the restriction $\Psi|_{\Delta}$ is an affine map. 

We are now ready to declare that our cycle $P$ is in fact given by $\mathbf{q}_\sharp \bar{P}$. It is immediate to see that 
\begin{align*}
\mathbb{M} (P) &\leq (1+ C (\bar \kappa+\eta)) \mathbb{M} (\bar{P}) \leq (1+C(\bar\kappa+\eta)) (\mathbb{M} (T) + \kappa),\\
\mathbb{F} (T-P) &\leq \mathbb{F} (T-\bar P) + \mathbb{F} (\bar P - P) \leq \kappa + C (\bar\kappa + \eta) (\mathbb{M} (\bar P))\\
&\leq \kappa + C (\bar \kappa +\eta) (\mathbb{M} (T) + \kappa),\\
{\rm spt}\, (P) &\subset B_{\kappa+\eta} ({\rm spt} (\bar P)) \subset B_{2\kappa+\eta} ({\rm spt}\, (T))\, .
\end{align*}
In particular, choosing $\kappa, \bar\kappa$ and $\eta$ appropriately, $P$ satisfies the three desired estimates ($a_0$), ($b_0$), and ($c_0$) in Proposition \ref{p:poly_approx_prescribedsing}. 

\medskip

Consider now the finite collection $\mathcal{P}$ simplices $\Gamma_i$ which are images through $\Psi$ of some $m$-dimensional simplex $\Delta_i$ of the triangulation of $K$. Some of these might have dimension strictly smaller than $m$ (which would mean that the affine map $\Psi|_\Delta$ does not have full rank). We then discard them from $\mathcal{P}$. Upon choosing an orientation for the $\Gamma_i$'s, we clearly have that 
\[
P = \sum_i \ell_i \psi_\sharp \llbracket \Gamma_i \rrbracket\, ,
\]
for an appropriate choice of the multiplicities $\ell_i$. 

We now can apply Proposition \ref{p:nello-scheletro-2} and find a triangulation $\mathcal{T}_3$ of $L$ which refines the triangulation $\mathcal{T}_2$ and with the property that each $\Gamma_i$ is the union of some elements in the $m$-dimensional skeleton of $\mathcal{T}_3$. The image through $\psi$ of $\mathcal{T}_3$ gives the desired triangulation $\mathcal{K}$ of $\mathcal{M}$ which satisfies the requirement ($d_0$) in Proposition \ref{p:poly_approx_prescribedsing}.

Finally, observe that there is an integral current $Z$ in $\mathbb R^N$ such that $T-\bar{P} = \partial Z$ and with $\spt (Z) \subset B_{2\kappa} (\mathcal{M})$. In particular $\mathbf{p}_\sharp Z$ provides an integral current in $\mathcal{M}$ such that $\partial \mathbf{p}_\sharp Z = T - \mathbf{p}_\sharp \bar{P}$. Given that $\mathbf{p}$ and $\mathbf{q}$ are close in the Lipschitz norm, there is a Lipschitz homotopy of the two maps which takes values in $B_{2\kappa} (\mathcal{M})$. Composing the latter homotopy with $\mathbf{p}$, we find a Lipschitz homotopy $\Phi$ between the two maps which takes values in $\mathcal{M}$: through the homotopy formula this map provides an integral current $Z'$ in $\mathcal{M}$ such that $\partial Z'= \mathbf{p}_\sharp \bar{P} - \mathbf{q}_\sharp \bar{P}$. In particular we conclude that $P$ is in the same homology class of $\mathbf{p}_\sharp \bar{P}$ and hence in the same homology class of $T$ in $\mathcal{M}$.

\subsection{Proof of Proposition \ref{p:poly_approx_prescribedsing}: second approximation}

Starting with the approximation $P$ and the triangulation $\mathcal{K}$ of the first part of Proposition \ref{p:poly_approx_prescribedsing} we now construct the approximation $P'$ of the second part.
This is done in two steps:

\medskip

{\bf Step 1. Regularization on $\mathcal{M}\setminus \mathcal{K}^{m-1}$.} In this first step we modify $P$ using the following algorithm.

We start by fixing, for the triangulation $\mathcal{K}$, a suitable polyhedral submanifold $K$ (\textit{cfr}. Definition \ref{d:PL-manifolds}) of some Euclidean space $\mathbb R^N$ and a piecewise smooth homeomorphism $\psi: K\to \mathcal{M}$ (\textit{cfr}. again Definition \ref{d:PL-manifolds}) which realizes the triangulation $\mathcal{K}$ in the sense that, for some suitable triangulation $\mathcal{T}$ of the polyhedron $K$, the following holds: for every cell $F$ of $\mathcal{K}$, its diffeomorphic preimage $\psi^{-1} (F)$ is a simplex of $\mathcal{T}$. The current $P$ is then given 
\[
P = \sum_{F\in \mathcal{F}^m} \beta_F \llbracket F \rrbracket\, ,
\]
where $\beta_F\in \mathbb Z$ and $\mathcal{F}^m$ is the collection of $m$-dimensional cells of $\mathcal{K}$. Without loss of generality we can assume that $\beta_F\geq 0$. For every cell $F$ with $\beta_F>0$ we consider $\Delta:= \psi^{-1} (F)$ and we let $\Gamma$ be an $(m+1)$-dimensional simplex of the triangulation $\mathcal{T}$ which contains $\Delta$. We will replace $\beta_F \llbracket F \rrbracket$ with $\sum_{j=1}^N \llbracket \psi (\Delta_j)\rrbracket$, where the $\Delta_j\subset \Gamma$ are diffeomorphic images of $\Delta \subset \Gamma$ with $\partial \Delta_j=\partial \Delta$ and $\Delta_{j'}\cap \Delta_{j''}=\partial \Delta$ for every $j'\neq j''$. In order to define the $\Delta_j$ we will use the following elementary lemma.

\begin{lem}\label{l:costruisco_funzione}
Consider the $m$-dimensional simplex $\Omega\subset \mathbb R^{m}$ which is the convex hull of $\{e_0, e_1, \ldots, e_m\}$, where $e_0=0$ and $e_1, \ldots, e_m$ is the standard basis. For each $i\in \{0, 1, \ldots, m\}$ let $F_i$ be the relative interior of the $m-1$-dimensional face of $\Omega$ spanned by $e_0, \ldots, e_{i-1}, e_{i+1}, \ldots, e_m$. In other words, $F_i$ consists of those points $p$ which can be written as convex combinations $\sum_j \lambda_j e_j$ with $\lambda_i=0$ and $\lambda_j>0$ for every $j\neq i$.

Then there is a Lipschitz function $f:\Omega \to \mathbb R$ and a neighborhood $V$ of $\bigcup_i F_i$ such that
\begin{itemize}
\item[(a)] $f$ is positive and smooth in the interior of $\Omega$;
\item[(b)] $f (x) = \dist (x, \partial \Omega)$ for every $x\in V$;
\item[(c)] ${\rm Lip}\, (f) \leq C$ for some constant $C=C(m)$.
\end{itemize}
\end{lem}

With Lemma \ref{l:costruisco_funzione} at hand, we are ready to define $\Delta_j$. First of all let $\{v_0, v_1, \ldots, v_m\}$ be the extremal points of $\Delta$, $\pi$ the $m$-dimensional linear space spanned by $\{v_1-v_0, \ldots , v_m-v_0\}$ and then let $v_{m+1}$ be the only unit vector orthogonal to $\pi$ with the property that $\frac{1}{2} \sum_i v_i + \gamma v_{m+1} \in \Gamma$ for every $\gamma$ sufficiently small. Let $A: \Delta \to \Omega$ be the affine map defined by $A (v_i)=e_i$. Choose then positive numbers $0<\delta_1 < \delta_2 < \ldots < \delta_{\beta_F}$ and define $\Delta_j$ as 
\[
\Delta_j = \{x+ \delta_j\, f(A(x)) v_{m+1} : x\in \Delta\}\, .
\]
Provided the $\delta_{\beta_F}$ is chosen sufficiently small, each $\Delta_j$ is contained in $\Gamma$.
Note moreover that, by construction, $\psi (\Delta_{j'})\cap \psi (\Delta_{j''})\subset \mathcal{K}^{m-1}$, the $(m-1)$-dimensional skeleton of $\mathcal{K}$.

We perform the above construction for all $F$'s with $\beta_F>0$. Upon enumerating them as $\{F^i\}$, we denote by $\Delta^i_j$ the corresponding $m$-dimensional cells in the polyhedron $K$, by $\delta^i$ the number $\delta_{\beta_{F_i}}$ and by $F^i_j$ their images through $\psi$. We note further that, choosing the $\delta_{\beta_{F_i}}$ sufficiently small, we can ensure that $F^i_j \cap F^{i'}_{j'} \subset \mathcal{K}^{m-1}$ for every choice of distinct pairs $(i,j)$ and $(i',j')$.

Consider now the integral current
\[
\tilde{P} := \sum_i \sum_j \llbracket F^i_j\rrbracket\, .
\]
Clearly $\partial (\sum_j \llbracket F^i_j\rrbracket) -\partial (\beta^i \llbracket F_i\rrbracket)=0$ and so $\sum_j \llbracket F^i_j\rrbracket - \beta^i \llbracket F_i \rrbracket$ is a cycle $T_i$. Moreover we can use the homotopy formula to ensure that $\mathbb{M} (T_i)$ is as small as needed provided $\delta^i$ is chosen comparably small. We can also ensure that $\mathcal{H}^m (F^i_j)$ is as close as needed to $\mathcal{H}^m (F)$ using the area formula. In particular we can conclude that, upon suitably choosing the $\delta^i$'s, 
\begin{itemize}
\item $\tilde{P}$ is in the same homology class as $P$ (in $\mathcal{M}$);
\item $\spt (\tilde{P}) \subset B_{2\varepsilon_c} (\spt (T))$;
\item $\mathbb{M} (\tilde{P}) \leq (1+2\varepsilon_c) \mathbb{M} (T)$;
\item $\mathbb{F} (T-\tilde{P}) \leq 2\varepsilon_c$;
\item $\spt (\tilde{P}) \setminus \mathcal{K}^{m-1}$ is smooth and is taken with multiplicity $1$ by $\tilde{P}$, more precisely:
\begin{itemize}
\item[(S)] for every $p\in \spt (\tilde{P}) \setminus \mathcal{K}^{m-1}$ there is a neighborhood $U$ of $p$ and a smooth oriented $m$-dimensional submanifold $\Lambda$ of $\mathcal{M} \cap U$ with boundary contained in $\mathcal{M} \cap \partial U$ and such that 
$\tilde{P} \res U = \llbracket \Lambda \rrbracket$. 
\end{itemize}
\end{itemize}
In the next (and final) step of the proof of Proposition \ref{p:poly_approx_prescribedsing} we will perturb $\tilde{P}$ by removing its $m-1$-dimensional singularities away from a small neighborhood of $\mathcal{K}^{m-2}$. But before coming to that, we provide the elementary proof of Lemma \ref{l:costruisco_funzione}.

\begin{proof}[Proof of Lemma \ref{l:costruisco_funzione}]
We denote by $\Omega^{m-2}$ the $m-2$-dimensional skeleton of $\partial \Omega$, namely the set of points which are convex combinations $\sum_i \lambda_i e_i$ where at least two among the coefficients $\lambda_i$ vanish, while we denote by $\pi_j$ the affine $m-1$-dimensional space which is formed by linear combinations $\sum_i \lambda_i e_i$ with $\lambda_j=0$. Moreover we denote by $A_j$ the affine function which vanishes on $\pi_j$ and coincides with $\dist (p, A_j)$ on the halfspace containing $\Omega$.

We then observe that there is a (sufficiently small) positive constant $\varepsilon_0$ and a (sufficiently large) positive constant $C_0$, both depending on the dimension $m$, such that the following holds.
\begin{itemize}
\item[(L)] On the open set $\{p\in \Omega: \dist (p, \pi_i)< \varepsilon_0\; \mbox{and} \; \dist (p, \pi_i) < C_0^{-1} \dist (p, \Omega^{m-2}) \}$ the function $\dist (p, \partial \Omega)$ coincides with the affine function $A_i$.  
\end{itemize}
We now let $V_k := \{p\in \Omega: 2^{-k-2} < \dist (p, \partial \Omega) < 2^{-k}\}$ for $k\geq 1$, while $V_0 := \{p\in \Omega : \dist (p, \partial \Omega) > \frac{1}{2}\}$ and we consider a partition of unity $\varphi_k$ subordinate to it with the property that $\|\nabla \varphi_k\|_{C^0} \leq C 2^k$. Finally, we let $\psi \in C^\infty_c (B_1)$ be a nonnegative mollifier with $\int \psi = 1$. 

The function $f: \Omega \to \mathbb R$ is then defined as 
\[
f := \sum_k \varphi_k \dist (\cdot, \partial \Omega) * \psi_{c_0 2^{-k}}
\]
for a sufficiently small constant $c_0$. Using (L) and the property that $A_i * \psi_\lambda = A_i$ for every choice of $\lambda$, we see immediately that $f$ coincides with $\dist (\cdot, \partial \Omega)$ in a neighborhood of $\bigcup_i F_i$. The positivity and smoothness of $f$ in $\Omega$ is in turn obvious. Finally, we can compute
\[
\nabla f = \sum_k \nabla \varphi_k \dist (\cdot, \partial \Omega) * \psi_{c_0 2^{-k}}
+ \sum_k \varphi_k \nabla (\dist (\cdot, \partial \Omega)*\psi_{c_0 2^{-k}}).
\]
The second summand is bounded by 
\[
\sum_k \varphi_k =1
\]
because the distance function is 1-Lipschitz. As for the first summand, given that $\sum_k \nabla \varphi_k = 0$, it equals  
\[
\sum_k \nabla \varphi_k (\dist (\cdot, \partial \Omega) * \psi_{c_0 2^{-k}} - \dist (\cdot, \partial \Omega))\,.
\]
For every fixed $p\in \Omega$, there is a $j$ such that $p$ is not in the support of $\varphi_k$ for $k > j+2$ and $k<j$. We can thus write 
\begin{align*}
& \left|\sum_k \nabla \varphi_k (p) (\dist (p, \partial \Omega) * \psi_{c_0 2^{-k}} - \dist (p, \partial \Omega))\right|\\
\leq & C 2^{\,j+2} \sum_{k=j}^{j+2} |\dist (p, \partial \Omega) * \psi_{c_0 2^{-k}} - \dist (p, \partial \Omega))| \leq C\, ,
\end{align*}
where we have used that $|\dist (p, \partial \Omega) * \psi_\lambda - \dist (p, \partial \Omega)| \leq \lambda$ for every $p$ with $\dist (p, \partial \Omega) > \lambda$.
\end{proof}

\medskip

{\bf Step 2. Removing the $(m-1)$-dimensional singularities.} Next consider an arbitrary face $F^k$ as in the previous subsection and let $\sigma_i$ be an arbitrary $(m-1)$-dimensional face of $F^k$. Fix a $\delta_c>0$: the goal is to modify $\tilde{P}$ in a neighborhood of $\sigma\setminus B_{\delta_c} (\mathcal{K}^{m-2})$ to a new current $P'$ in the same homology class, close to it in terms of mass, support and flat norm, and with the property that $P'$ is smooth in that neighborhood. The neighborhood in which we will perturb $\tilde{P}$ is of the form $B_\lambda (\sigma)\setminus B_{\delta_c} (\mathcal{K}^{m-2})$. First of all, given the structure of $\tilde{P}$ obtained in the previous subsection, if $\lambda$ is sufficiently small, there is an open subset $\Delta_i\subset \mathbb R^{m-1}$ and a smooth parametrization 
\[
\Phi : \Delta_i \times B_{\lambda_i}^{n+1} \to \mathcal{M}
\]
of a normal neighborhood $\mathcal{N}_i$ of $\sigma \setminus B_{\delta_c} (\mathcal{K}^{m-2})$ with thickness $\lambda_i$ and with the property that $\spt (\tilde{P}\setminus \mathcal{N}_i)$ can be described in the following way. There are a finite number of distinct unit vectors $w_1, \ldots , w_L \in \partial B_1^{n+1} \subset \mathbb R^{n+1}$ such that, if we let 
\[
\Lambda_\ell = \{(\sigma, s w_\ell) : \sigma \in \Delta_i, 0 < s < \lambda_i\}\, ,
\]
then $\tilde{P} \res \mathcal{N}_i = \sum_{\ell =1}^L \epsilon_\ell \Phi_\sharp \llbracket \Lambda_\ell \rrbracket$, where $\epsilon_\ell$ takes values in $\{-1,1\}$. Given that $\tilde{P}$ has no boundary in $\mathcal{N}_i$, we conclude that $L$ must be an even number $2\bar{L}$ and that exactly $\bar{L}$ among the numbers $\epsilon_\ell$ equal $1$, while the remaining equal $-1$. We can thus write 
\[
\tilde{P} \res \mathcal{N}_i = \sum_{\ell =1}^{\bar L} \Phi_\sharp \llbracket \Lambda_\ell \rrbracket
- \sum_{\ell = \bar L+1}^{2 \bar L} \Phi_\sharp \llbracket \Lambda_\ell \rrbracket\, .
\]
Consider now the oriented halflines $H_\ell = \{\lambda w_\ell : \lambda >0\}$ in $\mathbb R^{n+1}$. Upon reordering them, we can find disjoint smooth oriented curves $\gamma_\ell$ for $\ell \in \{1, \ldots , \bar L\}$ in $\mathbb R^{n+1}$ with the property that $\llbracket \gamma_\ell \rrbracket \res \mathbb R^{n+1} \setminus B_1= (\llbracket H_\ell \rrbracket  - \llbracket H_{\bar L + \ell}\rrbracket)\res \mathbb R^{n+1}\setminus B_1$. Furthermore we let $\tau_t: \mathbb R^{n+1} \to \mathbb R^{n+1}$ be the homothety $y\mapsto t y$ and denote by $\gamma_{\ell, t}$ the curve $\tau_t (H_\ell)$. We are now ready to define a replacement for $\tilde{P}\res \mathcal{N}_i$. We fix a smooth compactly supported function $\psi_i$ in $\mathbb R^{n+1}$ which is positive on $\Delta_i$ and vanishes on $\partial \Delta_i$, a small positive number $\kappa_i$, and we define 
\[
\Sigma^i := \left\{(x,y) : x\in \Delta_i, y \in \bigcup_{\ell} \gamma_{\ell, \kappa_i \psi_i (x)}\right\} \cap \Delta_i \times B^{n+1}_{\lambda_i}\, .
\]
Choosing $\kappa_i$ sufficiently small we can ensure that the current
\[
P^{i} := \Phi_\sharp \Sigma^i 
\]
satisfies $\partial P^i = \partial (\tilde{P}\res \mathcal{N}_i)$. Moreover we can make $\mathbb{F} (P^i - \tilde{P}\res \mathcal{N}_i)$ and $\mathbb M (P^i) - \mathbb M (\tilde{P}\res \mathcal{N}_i)$ smaller than any desired threshold by choosing $\kappa_i$ sufficiently small. Note finally that, clearly, $\Sigma_i$ is smooth in $\mathcal{N}_i$. 

We next enumerate all the $m-1$-dimensional faces $\sigma_i$ of all the $m$-dimensional faces $F^k$ as $\sigma_1, \sigma_2, \ldots , \sigma_N$. We choose our parameters in such a way that the sets $\mathcal{N}_i$ are pairwise disjoint. Our final current $P'$ will then be defined to be 
\[
P' := \sum_i P^i + \tilde{P} \res \mathcal{M} \setminus \bigcup_i \mathcal{N}_i\, .
\]
$P'\res \mathcal{M}\setminus B_{\delta_c}(\mathcal{K}^{m-2})$ is then smooth by construction and, choosing the parameters accordingly, $P'$ is homologous to $\tilde P$ and we achieve the desired estimates since we can make $\spt (P')$ arbitrarily close to $\spt (\tilde{P})$, $\mathbb{M} (P')$ arbitrarily close to $\mathbb{M} (\tilde{P})$, and $\mathbb{F} (P'-\tilde{P})$ arbitrarily small.

Finally, coming to the estimate on $\|P'\| (B_{\delta_c'} (\mathcal{K}^{m-2}))$ observe that we know:
\begin{align*}
&\|P'\| (\mathcal{M})\leq \|P\| (\mathcal{M}) + 2 \varepsilon_c\\
&\|P'\| (\mathcal{M}\setminus B_{\delta_c'} (\mathcal{K}^{m-2}))\geq\|P\| (\mathcal{M}\setminus B_{\delta_c'} (\mathcal{K}^{m-2}))\, .
\end{align*}
Hence
\begin{align*}
\|P'\| (B_{\delta'_c}(\mathcal{K}^{m-2})) &= \|P'\| (\mathcal{M}) - \|P'\| ( \mathcal{M}\setminus B_{\delta_c'} (\mathcal{K}^{m-2}))\\
&\leq \|P\| (\mathcal{M}) + 2\varepsilon_c - \|P\| (\mathcal{M}\setminus B_{\delta_c'} (\mathcal{K}^{m-2})) 
= \|P\| (B_{\delta_c'} (\mathcal{K}^{m-2})) + 2\varepsilon_c 
\end{align*}
for every $\delta_c'>\delta_c$. Hence $\delta_c'$ must be chosen small enough just to ensure that $\|P\| (B_{\delta_c'} (\mathcal{K}^{m-2})) \leq \varepsilon_c$.

\subsection{Proof of Proposition \ref{p:squash}} First of all we observe the following consequence of the area formula. 

\begin{lem}\label{l:squash-2}
Assume $\Phi$, $\gamma$, and $\varepsilon_a$ are as in Lemma \ref{l:Phi}. If $Z$ is any integer rectifiable current of dimension $m>k$, then
\begin{equation}\label{e:est-mass}
\mathbb{M} (\Phi_\sharp (Z\res B_{\delta_d} (\mathcal{K}^k)))\leq C (1+ \varepsilon_a)^k \gamma^{m-k} \|Z\| (B_{\delta_d} (\mathcal{K}^k))\, ,   
\end{equation}
where $C$ is a dimensional constant which depends only on $m$ and $n$. 
\end{lem}
\begin{proof}
Using the area formula, we have 
\begin{align}
\mathbb{M} (\Phi_\sharp (Z\res B_{\delta_d}(\mathcal{K}^k))) =
\int_{B_{\delta_d} (\mathcal{K}^k)} |d\Phi_p (\vec{Z} (p))|\, d\|Z\| (p)\, ,
\end{align}
where $\vec{Z} (p)$ is a unit simple $m$-vector orienting $Z$ at $p$. We can write $v=\pm v_1\wedge \ldots \wedge v_m$ for any orthonormal base of the approximate tangent space $V$ to $Z$ at $p$ and estimate 
\[
|d\Phi_p (\vec{Z} (p))| \leq \Pi_{i=1}^m |D\Phi_p (e_i)|\, .
\]
Now consider the space $W$ spanned by $e_1, \ldots , e_k$ and let $\mathbf{p}_V (W)$ be its orthogonal projection onto $V$. Clearly the dimension of $W' := \mathbf{p}_V (W)$ is at most $k$ and hence its orthogonal complement $W''$ in $V$ has dimension at least $m-k$. We can choose an orthonormal base of $V$ by completing an orthonormal base of $W''$. On the other hand any element of $W''$ is in the span of $e_{k+1}, \ldots, e_m$. In particular, we conclude that $|D\Phi_p (w'')|\leq \sqrt{m-k} \,\gamma |w''|$ for any vector $w''\in W''$. On the other hand the estimate $|D\Phi_p (v)|\leq \sqrt{m+n} \,(1+ \varepsilon_a) |v|$
holds for any vector $v\in T_p \mathcal{M}$, thus completing the proof of the estimate. 
\end{proof}

\begin{proof}[Proof of Proposition \ref{p:squash}]
We can assume, without loss of generality, that the homology class of $S$ is nontrivial, so that $\mathbb{M} (S)>0$.
The conclusion that $R$ and $S$ are homologous follows from the fact that they coincide outside $B_{\delta_d} (\mathcal{K}^{m-2})$. In particular $\spt (S-R)\subset B_{\delta_d} (\mathcal{K}^{m-2})$: since for $\delta_d$ smaller than a constant $c (\mathcal{K})$ the latter deformation retracts onto $\mathcal{K}^{m-2}$, whose $m$-dimensional homology is trivial, it follows that $S-R$ is homologically trivial. 

We now let $\varepsilon_d$ and $\eta_d$ be given as in the statements. We further fix $\bar{\varepsilon}_d$, whose choice will be specified later (but which will depend only on $\varepsilon_d$), and apply Lemma \ref{l:Phi} with $\varepsilon_a=\bar\varepsilon_d$ and $\eta_a=\eta_d$. We therefore get the parameter $\delta_a=:\delta_d$ (which will be the one of the conclusion of the proposition) and, after fixing yet another $\gamma$ (whose choice will now be dependent on $R$), we get the map $\Phi$ satisfying the requirements of Lemma \ref{l:squash-2}. The requirements (3) and (4) of Proposition \ref{p:squash} are then satisfied by construction.  
Moreover estimate (2) follows from the isoperimetric inequality and from (1) and (3). Indeed there is an integral current $T$ such that $\partial T = S-R'$ and 
\[
\mathbb{M} (T) \leq C \left(\mathbb{M} (S-R')\right)^{\frac{m+1}{m}}
\]
with $C= C (\mathcal{M})$. Using (1) and (3) we then estimate
\begin{align}
\mathbb{M} (S-R') &= \|S-R'\| (B_{\eta_d} (\mathcal{K}^{m-2})) \leq \|S\| (B_{\eta_d} (\mathcal{K}^{m-2}))
+ \|R'\| (B_{\eta_d} (\mathcal{K}^{m-2})) \nonumber\\
&= \|S\| (B_{\eta_d} (\mathcal{K}^{m-2})) + \mathbb{M} (R') - \|S\| (\mathcal{M}\setminus B_{\eta_d} (\mathcal{K}^{m-2}))\nonumber\\
&= 2 \|S\| (B_{\eta_d} (\mathcal{K}^{m-2})) + \mathbb{M} (R') - \mathbb{M} (S)\nonumber\\
&\leq 2 \|S\| (B_{\eta_d} (\mathcal{K}^{m-2})) + \varepsilon_d \mathbb{M} (S)\,.
\end{align}
It remains to prove (1). Note that 
\begin{align}
\mathbb{M} (R') &\leq \mathbb{M} (\Phi_\sharp (R\res B_{\delta_d} (\mathcal{K}^{m-2}))) + \mathbb{M} (\Phi_\sharp (R\res \mathcal{M}\setminus B_{\delta_d} (\mathcal{K}^{m-2})))\nonumber\\
&= \mathbb{M} (\Phi_\sharp (R\res B_{\delta_d} (\mathcal{K}^{m-2}))) + \mathbb{M} (\Phi_\sharp (S\res \mathcal{M}\setminus B_{\delta_d} (\mathcal{K}^{m-2})))\nonumber\\
&\leq \mathbb{M} (\Phi_\sharp (R\res B_{\delta_d} (\mathcal{K}^{m-2}))) + ({\rm Lip}\, \Phi)^m \mathbb{M} (S)\nonumber\\
&\leq \mathbb{M} (\Phi_\sharp (R\res B_{\delta_d} (\mathcal{K}^{m-2}))) + (1+\bar \varepsilon_d)^m \mathbb{M} (S)\, .
\end{align}
Hence we apply Lemma \ref{l:squash-2} and infer 
\[
\mathbb{M} (R') \leq C (1+\bar\varepsilon_d)^{m-2} \gamma^2 \mathbb{M} (R) + (1+\bar \varepsilon_d)^m \mathbb{M} (S)\,.
\]
Next we fix $\bar \varepsilon_d$ so that $(1+\bar\varepsilon_d)^m=1+\frac{\varepsilon_d}{2}$, and then we choose $\gamma$ sufficiently small so that 
\[
C (1+\bar\varepsilon_d)^{m-2} \gamma^2 \mathbb{M} (R) \leq \frac{\varepsilon_d}{2} \mathbb{M} (S)\, .
\]
Note that the choice of $\gamma$, unlike that of $\varepsilon_d$, will indeed depend on $R$ and $S$. 
\end{proof}

\section{Proof of the main theorem}\label{s:proof}
We begin with some preliminary lemmas. 

\begin{lemma}\label{l:m-5/n+4}
    Let $\mathcal{M}$ be as in Assumption \ref{a:1}, $\mathcal{K}$ a smooth triangulation of $\mathcal{M}$, $k\in \{0, \ldots , m+n-1\}$ and $U_\delta (\mathcal{K}^k)$ as in Lemma \ref{l:spigoli-allisciati}. Then the complement of $U_\delta(\mathcal{K}^k)$ is homotopy equivalent to a complex of dimension $m+n-k-1$.
\end{lemma}

\begin{proof} 
First we note that the complement of $U_\delta(\mathcal{K}^k)$ is a deformation retract of the complement of $\mathcal{K}^k$ by Lemma \ref{l:spigoli} and Lemma \ref{l:spigoli-allisciati}, and therefore $\mathcal{K} \setminus U_\delta(\mathcal{K}^k)$ is homotopy equivalent to $\mathcal{K} \setminus \mathcal{K}^k$; we then denote\footnote{With an abuse of notation between the simplicial complex and the $t$-image of the geometric realization of the simplicial complex itself.} $$\mathcal{K}^k_c:= \mathcal{K}\setminus\mathcal{K}^k.$$

Now we show that $\mathcal{K}^k_c$ is homotopy equivalent to a complex of dimension $m+n-k-1$, and in particular that \begin{equation}\label{e:scheletroduale}\mathcal{K}^k_c \sim \mathcal{K}_*^{m+n-k-1},\end{equation}
where $\mathcal{K}_*$ is the dual cell complex of the triangulation $\mathcal{K}$, \textit{cfr.} \cite[\S 64]{MunkresElements}. To this aim, we first show the following:
\begin{equation}\label{e:suddivisionebari}
    \mathcal{K}^k_c \sim \text{Bs}(\mathcal{K}) - \text{Bs}(\mathcal{K}^k),
\end{equation} where $\text{Bs}(\cdot)$ is the barycentric subdivision and the operation $\text{Bs}(\mathcal{K})-\text{Bs}(\mathcal{K}^k)$ represents all simplexes of $\text{Bs}(\mathcal{K})$ that are disjoint from $\text{Bs}(\mathcal{K}^k)$. By definition, we have that $$\mathcal{K}^k_c= \mathcal{K}\setminus \mathcal{K}^k= \text{Bs}(\mathcal{K}) \setminus \text{Bs}(\mathcal{K}^k).$$ Hence \eqref{e:suddivisionebari} follows just by noticing that, in terms of simplicial complexes, $\text{Bs}(\mathcal{K}^k)$ is a full subcomplex of the complex $\text{Bs}(\mathcal{K})$, \textit{i.e.} every simplex of $\text{Bs}(\mathcal{K})$ whose vertices are in $\text{Bs}(\mathcal{K}^k)$ is itself in $\text{Bs}(\mathcal{K}^k)$, and by applying \cite[Lemma 70.1]{MunkresElements}. Since the complex $\text{Bs}(\mathcal{K}) - \text{Bs}(\mathcal{K}^k)$ corresponds exactly to the $m+n-k-1$-skeleton $\mathcal{K}^{m+n-k-1}_*$ of the dual cell structure of $\mathcal{K}$, \eqref{e:scheletroduale} follows, and hence the final result.
\end{proof}

\begin{lem}\label{l:m+4_equivalenza}
Let $h : T(\widetilde{\gamma}^n) \rightarrow \mathbf{K}(\mathbb{Z},n)$ be a map representing the Thom class $u \in \mathbf{H}^n(T(\widetilde{\gamma}^n), \mathbb{Z})$. Then $h$ is an $n+4$-equivalence, for all positive integers $n$.
\end{lem}

\begin{proof}
The spaces are the same for $n\in \{1,2\}$: $T(\widetilde{\gamma}^1)$ is homotopy equivalent to the circle $S^1$, which is a realization of $\mathbf{K}(\mathbb{Z},1)$, while $T(\widetilde{\gamma}^2)$ is homotopy equivalent to the infinite complex projective space $\mathbb{C}\mathbb{P}(\infty)$, of type $\mathbf{K}(\mathbb{Z},2).$ Hence, we can assume $n\geq 3.$ Towards an application of Theorem \ref{t:II.6}, we recall the computations of the cohomology rings of $\EilenbergML$ and of the classifying space $\BSOn$, for any group coefficient $\mathbb{Z}_p$.

By Serre's computations using spectral sequences of fibre spaces, the cohomology $\mathbf{H}^{n+i}\big(\mathbf{K}(\mathbb{Z},n)\big)$ with $\mathbb{Z}_2$ coefficients is generated by the Steenrod squares $Sq^2, Sq^3$ (and $Sq^4$ if $n\geq 4$) for $i\leq 4$, see \cite[Théorème 3]{Serre53}. By calculations of Cartan with coefficients in $\mathbb{Z}_3$, the cohomology $\mathbf{H}^{n+i}(\mathbf{K}(\mathbb{Z},n))$ is generated by $\mathcal{P}^1_3$ in dimensions less than or equal to $n+4$, while for $\mathbb{Z}_p$ coefficients with prime $p>3$ there are no generators between dimension $n$ and dimension $n+8$, see for example \cite[Chapitre II, \S 8, \S 9]{Thom54} or \cite[\S 10.5]{DubrovinFomenkoNovikovIII}.

The cohomology ring of $\BSOn$ with coefficients in $\mathbb{Z}_2$ is generated by the Stiefel-Whitney classes $w_2, \dots, w_n$ of $\widetilde{\gamma}^n$, \textit{cfr.} Proposition \ref{p:BSOn_mod2}, that is $$\mathbf{H}^{*}\big(\BSOn, \mathbb{Z}_2\big)= \mathbb{Z}_2\big[w_2,w_3,w_4 \dots w_n\big].$$

For odd primes and in dimensions $i\leq 5$, we have that, \textit{cfr.} Proposition \ref{p:BSOn_modp}, $$ \begin{aligned} 
\mathbf{H}^{*}\big(\BSOn,\mathbb{Z}_p\big)&= \mathbb{Z}_p\big[p_1\big] \quad \quad \, \,\, \text{if } n \neq 4,\\ 
\mathbf{H}^{*}\big(\mathbf{BS O}(4), \mathbb{Z}_p\big)&= \mathbb{Z}_p\big[p_1, e\big], \quad\, \text{if } n =4.\end{aligned}$$

For every $p$ prime, let $\Phi_p$ denote the Thom isomorphism between $\mathbf{H}^i(\BSOn,\mathbb{Z}_p)$ and $\widetilde{\mathbf{H}}^{n+i}(T(\widetilde{\gamma}^n), \mathbb{Z}_p)$ and $u_p$ the Thom class. Since by Proposition \ref{p:stiefelsteenrod} and \cite[Theorem 19.7]{MilnorStasheff} we have\footnote{Denoting with a slight abuse $p_1$ for $p_1$ reduced mod $3$.} that $\Phi_2(w_i)= Sq^i(u_2)$ and $\Phi_3(p_1)=\mathcal{P}_3^1(u_3)$, it follows that, for any group coefficient $\mathbb{Z}_p$, the induced map in cohomology $$h^*: \mathbf{H}^{n+i}\big(\mathbf{K}(\mathbb{Z},n),\mathbb{Z}_p\big) \rightarrow \mathbf{H}^{n+i}(T(\widetilde{\gamma}^n),\mathbb{Z}_p)$$ is an isomorphism for dimensions less than or equal to $n+3$ and a monomorphism in dimension $n+4$. Since $\mathbf{K}(\mathbb{Z},n)$ and $T(\widetilde{\gamma}^n)$ are simply connected, by Theorem \ref{t:II.6}, we conclude that $$\pi_k\big(\mathbf{K}(\mathbb{Z},n), T(\widetilde{\gamma}^n)\big)=0 \quad \text{for }k\leq n+4,$$ ending the proof.\end{proof}

\begin{remark}\label{r:rappresentofinomleq4}
 Lemma \ref{l:m+4_equivalenza} shows that in particular that, for dimension $m\in \{1,2,3,4\}$ and for any codimension $n\in \mathbb{N}\setminus\{0\}$, every homology class $\tau \in \mathbf{H}_m({\mathcal{M},\mathbb{Z}})$ is represented by an embedded smooth submanifold $\Sigma$ in $\mathcal{M}$. It is important to remark that, by \cite[p.56, footnote 9]{Thom54}, we also know that every homology class of $\mathbf{H}_m(\mathcal{M},\mathbb{Z})$ for $m\leq 6$ is representable by a smooth submanifolds, due to the vanishing of the obstruction of the corresponding Poincaré dual $x$, $St_3^5(x)$, where $St_3^5$ represents (up to a sign) the following cohomology operations, \textit{cfr.} Remark \ref{r:St_pi}, $$St_3^5=\beta^* \circ \mathcal{P}^1_3 \circ \theta_3:\mathbf{H}^*(\mathcal{M},\mathbb{Z}) \rightarrow \mathbf{H}^{*+5}(\mathcal{M},\mathbb{Z}).$$

\end{remark}

\begin{proof}[Proof of Theorem \ref{t:1}] 
Fix $ \varepsilon_c > 0$, whose choice will be specified later, and an integral current $T$ in a homology class $\tau\in \mathbf{H}_m (\mathcal{M}, \mathbb Z)$. First of all apply Proposition \ref{p:poly_approx_prescribedsing} to find a sufficiently small $\delta_c'>0$, a suitable triangulation $\mathcal{K}$ of the manifold and a new integral current $P'=:S$ with the property that $S$ is in the same homology class of $T$ and the following facts hold:
\begin{itemize}
\item $\mathbb{M} (S) \leq \mathbb{M} (T) + 3\varepsilon_c$ and $\mathbb{F} (S-T) < 3 \varepsilon_c$;
\item $\|S\| (B_{\delta_c'} (\mathcal{K}^{m-2})) \leq 3\varepsilon_c$;
\item $B_{\delta_c'} (\mathcal{K}^{m-2})$ is homotopy equivalent to $\mathcal{K}^{m-2}$;
\item $S\res \mathcal{M} \setminus B_{\delta_c} (\mathcal{K}^{m-2})= \llbracket \Gamma \rrbracket$ for a smooth submanifold $\Gamma$.
\end{itemize}
We observe the following important fact: if we first choose $\varepsilon_c$, then $\delta_c$, $\delta_c'$ and $\frac{\delta_c}{\delta_c'}$ can all be made smaller than any desired constant, while the triangulation is instead kept fixed (because it depends only on $\varepsilon_c$). 

We have now fixed a triangulation $\mathcal{K}$ and we can therefore fix constant $C_0$ and $\bar\delta$ so that Lemmas \ref{l:spigoli} and \ref{l:spigoli-allisciati} apply. We now require that $V_{{\delta}_c/2}(\mathcal{K}^{m-2})  \subset \subset  B_{\delta_c'} (\mathcal{K}^{m-2})$ for some ${\delta}_c/2 << \delta_c'$.  Hence we apply Lemma \ref{l:spigoli-allisciati} (where $\delta'<\delta$ corresponds here to $\delta_c/2<\tilde{\delta}/2)$ to find a $U_{\tilde{\delta}/2} (\mathcal{K}^{m-2})$ suitably close to $V_{\tilde{\delta}/2} (\mathcal{K}^{m-2})$. We will want that $B_{\delta_c}(\mathcal{K}^{m-2}) \subset U_{\tilde{\delta}/2} (\mathcal{K}^{m-2}) \subset V_{\tilde{\delta}/2} (\mathcal{K}^{m-2}) \subset V_{\tilde{\delta}} (\mathcal{K}^{m-2}) \subset B_{\delta_c'} (\mathcal{K}^{m-2})$. This step requires to take $\frac{\delta_c}{\delta_c'}$ sufficiently small and $\tilde \delta < \delta'_c$. Define now $\Omega := \mathcal{M} \setminus U_{\tilde{\delta}/2} (\mathcal{K}^{m-2})$.

The current $\llbracket \Gamma \rrbracket$ obtained from Proposition \ref{p:poly_approx_prescribedsing} is (when restricted to $\Omega$ and not relabelled) a smooth compact oriented submanifold of $\Omega$ with $\partial \Gamma \subset \partial \Omega$, provided $\partial \Omega$ is transversal to $\Gamma$, which can be ensured via a small smooth perturbation. Denoting by $x \in \mathbf{H}^{n}(\mathcal{M})$ the Poincaré dual of $\tau$, note that its restriction $x_{|\Omega} \in \mathbf{H}^n(\Omega)$ to $\Omega$ is the relative Poincaré dual of a relative homology class which is represented by the smooth compact embedded submanifold $\Gamma \subset \Omega$ with boundary $\partial \Gamma = \Gamma \cap \partial \Omega$. Hence, by Theorem \ref{t:Thomboundary}, there exists a map $$F: \Omega \rightarrow T(\widetilde{\gamma}^n)$$ such that $F^*(u) = x_{|\Omega}$; in addition, $F$ is smooth and transverse on $\BSOn$ (and such that $F_{|\partial \Omega}$ is also transverse), so that $F^{-1}(\BSOn)=\Gamma$, which is the smooth part of $S$. 

We then take $\delta$ sufficiently small so that $\Omega \subset \mathcal{M}\setminus U_\delta (\mathcal{K}^{m-5})$ for the $U_\delta$ given in Lemma \ref{l:spigoli-allisciati}. Then, by Lemma \ref{l:m-5/n+4} we have that $\mathcal{M} \setminus U_\delta(\mathcal{K}^{m-5})$ is homotopy equivalent to a complex of dimension $n+4$. Denote $$Q:=\mathcal{M} \setminus U_\delta(\mathcal{K}^{m-5}).$$
Given the $n$-dimensional cohomology class $x \in \mathbf{H}^{n}(\mathcal{M})$ which is the Poincaré dual of $\tau$, we consider its restriction to $Q$, that is $x_{|Q} \in \mathbf{H}^n(Q)$; note that $x_{|Q}$ can be represented by a continuous map $$g:Q \rightarrow \mathbf{K}(\mathbb{Z},n)$$ (in a suitable homotopy class of continuous maps) pulling-back the fundamental class of $\mathbf{K}(\mathbb{Z},n)$ to itself, \textit{i.e.} $g^*(\iota)=x_{|Q}.$ By Lemma \ref{l:m+4_equivalenza} and Proposition \ref{p:Spanier7622}, there exists a map $f: Q \rightarrow T(\widetilde{\gamma}^n)$ such that the diagram commutes, \textit{i.e.} $f$ pulls-back the universal Thom class to $x_{|Q}$.

\[
\xymatrix{
 & T(\widetilde{\gamma}^n) \ar[d]^{h} \\
Q \ar[r]_{g} \ar[ur]^{f} & \mathbf{K}(\mathbb{Z},n)
}
\]

By the same construction of the second part of the proof of Theorem \ref{t:Thomboundary}, we can assume without loss of generality that $f$ is smooth throughout $Q \setminus f^{-1}(U(\infty))$ and transversal to (a sufficiently high dimensional approximation of) the zero cross-section $\BSOn \subset T(\widetilde{\gamma}^n)$, with $\partial f$ also transversal to it. Hence, $f^{-1}(\BSOn)$ is a compact smooth $m$-dimensional embedded submanifold, with boundary contained in $\partial Q$; denote it as $$\mathcal{N}:=f^{-1}(\BSOn).$$ Moreover, $\mathcal{N}$ represents the relative Poincaré dual of $x_{|Q}$, which equals $j_*(\tau) \in \mathbf{H}_m(Q,\partial Q),$ where $j_*:\mathbf{H}_m(\mathcal{M}) \rightarrow \mathbf{H}_m(Q,\partial Q).$ 

We next wish to extend $\llbracket\mathcal{N}\rrbracket$ (which is an integral current in $\mathcal{M}$) to an integral current $N$ with the property that $N\res \mathcal{M}\setminus \mathcal{K}^{m-5}$ is a smooth submanifold with multiplicity $1$ and $N\res Q = \llbracket \mathcal{N}\rrbracket$. First of all, because $\mathcal{N}$ is transversal to $\partial Q$, we can extend it to a smooth submanifold over the union $Q'$ of $Q$ with any smooth collaring extension of $\partial Q$. We can then use Lemma \ref{l:spigoli-allisciati} to find such an extension $Q'$ (which consists of $Q\cup \mathcal{C}$, where $\mathcal{C}$ is the smooth tubular neighborhood in Lemma \ref{l:spigoli-allisciati}) containing $\mathcal{M}\setminus V_{\delta'} (\mathcal{K}^{m-5})$ for some $\delta'<\delta$ positive. Since $\mathcal{N}$ intersects $\partial Q$ transversally, we can extend to a smooth submanifold of $Q'$ with boundary in $\partial Q'$, meeting $\partial Q'$ transversally. With abuse of notation this extension is still denoted by $\mathcal{N}$. We can now use the map $\Phi$ of Lemma \ref{l:second-Phi} and set $$N:=\Phi_\sharp \llbracket \mathcal{N}\rrbracket.$$ The latter current is integer rectifiable because $\Phi$ is Lipschitz (and, in particular, $N$ has finite mass). Given that $\Phi$ is a diffeomorphism over $\mathcal{M}\setminus \Phi^{-1} (\mathcal{K}^{m-5}) \subset \mathcal{M}\setminus V_{\delta'} (\mathcal{K}^{m-5})$, then $N\res \mathcal{M}\setminus \mathcal{K}^{m-5} = \llbracket \Sigma \rrbracket$ for some smooth submanifold $\Sigma$. Moreover $\spt (\partial N)\subset \mathcal{K}^{m-5}$ and in particular, by Federer flatness theorem, $\partial N = 0$, namely $N$ is a cycle. 

Consider now the two maps $F: \Omega \rightarrow T(\widetilde{\gamma}^n)$ and $f: Q \rightarrow T(\widetilde{\gamma}^n)$ such that $F^{-1}(\BSOn)= \Gamma$ and $f^{-1}(\BSOn)= \mathcal{N}\cap Q$. If we consider the restriction of $f$ to $\Omega \subset Q$, we obtain a new map $f|_{\Omega}: \Omega \rightarrow T(\widetilde{\gamma}^n)$ that pulls-back the universal Thom class to $x_{|\Omega}$. By Lemma \ref{l:m-5/n+4} we observe that $\Omega$ has the homotopy type of an $(n+1)$-complex, so that by Corollary \ref{c:Whitehead_Homotopyclasses} we can conclude that $F$ and $f_{|\Omega}$ are homotopic: the homotopy can be taken smooth by \cite[Lemme IV.5]{Thom54}. In particular, we define the smooth homotopy $H:[0,1] \times \Omega$ such that $H(0,x)= f_{|\Omega}(x)$ and $H(1,x)= F(x)$. In a small collar neighborhood $\mathcal{C}$ of $\partial \Omega$ inside $\Omega$, which we identify with $\partial \Omega\times (0,1]$, we then glue the maps $f$ and $F$ together. Using the notation $x=(y,s)\in \mathcal{C}$ and after defining a smooth function $\varphi$ on $[0,1]$ which is identically equal to $0$ in a neighborhood of $0$ and identically equal to $1$ in a neighborhood of $1$, we set  
\begin{equation}\label{e:homotopy}
\widehat{f}(x):=\begin{cases}
       F(x)  & \quad if  \, \, x\in \Omega\setminus \mathcal{C},\\    
       H\left(x,\varphi (s)\right)  & \quad if\,\,  x\in \mathcal{C},   \\
        f(x)  & \quad if\,\, x\in Q\setminus \Omega  
\end{cases}\end{equation}

Since $T(\widetilde{\gamma}^n) \setminus \{\infty\}$ is a smooth submanifold, it follows from \cite[Proposition 2.3.4 (ii)]{Wall} that we can find $\widehat{f}:Q\rightarrow T(\widetilde{\gamma}^n)$, not relabelled, which is smooth throughout $Q \setminus f^{-1}(U(\infty))$, coincides with $f(x)$ in a neighborhood of $\partial Q$ and with $F$ on $\mathcal{M}\setminus V_{\tilde{\delta}} (\mathcal{K}^{m-2})$.
%; the approximation can be taken close enough so that $\widehat{f}$ is in the same homotopy class. C: Non serve!
Analogously, by \cite[Proposition 4.5.10]{Wall}, we can perturb $\widehat{f}$ so that it is transverse to $\BSOn$ and coinciding with $f(x)$ in a neighborhood of $\partial Q$ and with $F$ on $\mathcal{M}\setminus V_{\tilde{\delta}} (\mathcal{K}^{m-2})$.
%; again, the approximation can be taken close enough so that $\widehat{f}$ is in the same homotopy class. Come sopra: non serve.

Consider now the submanifold $\Sigma'$ of $\mathcal{M}\setminus \mathcal{K}^{m-5}$ which consists of:
\begin{itemize}
    \item $\Sigma$ in $V_{\delta'} (\mathcal{K}^{m-5})\setminus \mathcal{K}^{m-5}$;
    \item $\mathcal{N}$ on $U_\delta (\mathcal{K}^{m-5})\setminus V_{\delta'} (\mathcal{K}^{m-5})$;
    \item $\widehat{f}^{-1} (\BSOn)$ on $Q$.
\end{itemize}
This is a smooth submanifold because:
\begin{itemize}
    \item $f$ and $\widehat{f}$ coincide in a neighborhood of $\partial Q$ and hence $\widehat{f}^{-1} (\BSOn)$ coincides with $\mathcal{N}$ in a neighborhood of $\partial Q$;
    \item $\Sigma = \Phi (\mathcal{N}) = \mathcal{N}$ in a neighborhood of $\partial V_{\delta'} (\mathcal{K}^{m-5})$.
\end{itemize}
    Moreover, $R= \llbracket \Sigma'\rrbracket$ is an integer rectifiable current with finite mass and such that $\spt (\partial R)\subset \mathcal{K}^{m-5}$; in particular it is a cycle by Federer's flatness theorem. Observe also that $R-S$ is supported, by construction, in $V_{\tilde{\delta}} (\mathcal{K}^{m-2})$, which is homotopy equivalent to $\mathcal{K}^{m-2}$, and thus has trivial $m$-homology. In particular $R-S$ is a boundary, namely $R$ and $S$ belong to the same homology class. 

We now apply Proposition \ref{p:squash} to $S$ and $R$, noticing that the $\varepsilon_d$ in Proposition \ref{p:squash} is a parameter to be chosen in terms of the $\varepsilon$ of the statement of Theorem \ref{t:1}, and the $\eta_d$ in Proposition \ref{p:squash} is $\delta'_c$ here. This gives us a parameter $\delta_d$, which depends on $\varepsilon_d$ and $\delta'_c$. In turn we impose that $\tilde{\delta} \leq \delta_d$ so that we can apply Proposition \ref{p:squash}. Since $\varepsilon_d$ will be specified only in terms of $\mathbb{M}(T)$ and of $\varepsilon$ in the statement of Theorem \ref{t:1}, while $\delta'_c$ depends on $\varepsilon_c$, which will also be specified only in terms of $\mathbb{M}(T)$ and $\varepsilon$ in the statement of Theorem \ref{t:1}, the parameter $\tilde{\delta}$ can be taken smaller than $\delta_d$. We can then find a current $R':=\Phi_\sharp R$ for a smooth diffeomorphism $\Phi$ isotopic to the identity such that  
\begin{align}
\mathbb{M} (R') \leq (1+\varepsilon_d)\, \mathbb M (S) \leq (1+\varepsilon_d) (\mathbb{M} (T) + 3\varepsilon_c)\,\nonumber .
\end{align}
We therefore conclude that $R'$ is homologous to $R$, hence to $S$, and therefore to $T$. Moreover, if we choose
\begin{align}
&\varepsilon_d \,(3+\mathbb{M} (T)) < \frac{\varepsilon}{2} \qquad \mbox{and} \qquad
\,3\varepsilon_c < \frac{\varepsilon}{2}\,\nonumber ,
\end{align}
then $\mathbb{M} (R') \leq \mathbb{M} (T) + \varepsilon$. Finally
\begin{align}
\mathbb{F} (T-R') &\leq 3\varepsilon_c + \mathbb{F} (S-R') 
\leq 3\varepsilon_c + C (\varepsilon_d
\,\mathbb{M} (S) + 2\|S\| (B_{\delta'_c}(\mathcal{K}^{m-2})))^{\frac{m+1}{m}}\nonumber\\
&\leq 3\varepsilon_c + C (\varepsilon_d (\mathbb{M} (T) +\varepsilon_c) + 6\varepsilon_c)^{\frac{m+1}{m}}\,\nonumber.
\end{align}
Therefore it is clear that a suitable choice of $\varepsilon_d$ and $\varepsilon_c$ depending only on $\mathbb{M} (T)$ and $\varepsilon$ suffices to how $\mathbb{F} (T-R') \leq \varepsilon$. 

The proof of part $(3)$ of Theorem \ref{t:1} is analogous; by assumption we know that $\tau$ is represented by a smooth closed submanifold $\Sigma$ and hence, by Theorem \ref{t:Thomclosed} there exists a map $g:\mathcal{M} \rightarrow T(\univob)$ which pulls-back the universal Thom class $u \in \mathbf{H}^n(T(\widetilde{\gamma}^n), \mathbb{Z})$ to the Poincaré dual of $\tau$. Substituting in the previous steps the map $f$ with this new map $g$, defined over the whole ambient space $\mathcal{M}$, and defining a similar homotopy as that one in \eqref{e:homotopy}, the result follows by applying Proposition \ref{p:squash} to $S$ and $\llbracket \Sigma \rrbracket$, where $S$ is the integral cycle denoted $P'$ in Proposition \ref{p:poly_approx_prescribedsing}.

\end{proof}

\section{Optimality of the main theorem}\label{s:optimality}

The codimension 5 construction in Theorem \ref{t:1} is the best possible result in full generality, as shown by Theorem \ref{t:Thom_innatelysingular}. 

We start by recalling Thom's example of an integral homology class of dimension 7 in an orientable smooth manifold of dimension 14 which is not realizable by means of a submanifold, \textit{cfr.} \cite[Théorème III.9]{Thom54}\footnote{We remark that dimension 14 of the ambient space is not crucial: this example can be easily adapted to the lowest possible dimension allowed, that is dimension 10; we also refer to \cite{Hankecycles} for an example of a $7$ dimensional integral homology class which does not admit a smooth representative in a $10$ dimensional manifold with torsion-free homology.}.

\begin{example}[Thom]\label{ex:ThomLensSpaces}
For $i=1,2$ consider the lens space $L_i:=S^7/\mathbb{Z}_3$, which is the orbit space of the $7$-sphere with the free action of $\mathbb{Z}_3$ generated by the rotation. Let $v_i$ be generator of $\mathbf{H}^1(L_i,\mathbb{Z}_3)\simeq \mathbb{Z}_3$ and call $u_i=\beta_3(v_i) \in \mathbf{H}^2(L_i,\mathbb{Z}_3)\simeq \mathbb{Z}_3$. Consider the smooth oriented $14$-dimensional manifold $L:=L_1\times L_2$ and the following cohomology class, where the powers and $\cdot$ denote the cup product (seeing $\mathbf{H}^*(L_i)$ as embedded in $\mathbf{H}^*(L)$): $$y=u_1 \cdot v_2 \cdot(u_2)^2 - v_1 \cdot (u_2)^3 \in \mathbf{H}^7(L,\mathbb{Z}_3).$$ Note that $y$ is actually the reduction mod $3$ of an integral cohomology class, since $y=\beta_3(v_1 \cdot v_2 \cdot (u_2)^2)$ and hence $y=\theta_3(x)$, with $x \in \mathbf{H}^7(L,\mathbb{Z})$ given by $x =\beta^*(v_1 \cdot v_2 \cdot (u_2)^2)$, \textit{cfr.} Remark \ref{r:St_pi} for the notation. Denoting by $z \in \mathbf{H}_7(L,\mathbb{Z})$ its Poincaré dual homology class, we see that $z$ cannot be realized in $L$ by a submanifold since $$St_3^5(x)=\beta^*\circ \mathcal{P}_3^1(y)=\beta^*((u_1)^3 \cdot v_2 \cdot (u_2)^2) = (u_1 \cdot u_2)^3 \neq 0.$$
\end{example}

\begin{remark}
We remark that the obstruction to realizability comes from a cohomology operation mod $3$ and since $y \in \mathbf{H}^7(L,\mathbb{Z}_3)$, then the Poincaré dual of $3y$ can be realized by a submanifold. In general, it is a theorem of Thom, \textit{cfr.} \cite[Théorème II.29]{Thom54}, that for every integral homology class $z\in \mathbf{H}_k(\mathcal{M},\mathbb{Z})$ of a closed oriented manifold there exists a non-zero integer $N$ such that the class $Nz$ is realizable by a submanifold.
\end{remark}

Example \ref{ex:ThomLensSpaces} is the first example of \emph{innately singular} homology classes: from a geometric point of view, it represents a codimension $5$ non-removable singularity which is the geometric analogue of the algebraic obstruction given by the dual $3$-torsion cohomology operation $St_3^5.$ In particular, Thom's innately singular class can be represented by a $7$-dimensional cycle $T$ with a $2$-dimensional stratum of singularities, \textit{i.e.} a closed (equisingular) $2$-dimensional manifold $\mathcal{S}_T$ whose neighborhood is homeomorphic to a product $$\mathcal{S}_T \times C(\CP),$$ where $C(\CP)$ denotes the cone over $\CP$; the innate nature of these singularities turns out to be intrinsically linked to the well-known fact that $\CP$ does not bound any compact oriented smooth $5$-dimensional manifold, as observed in \cite{SullivanLiverpool}.

This geometric description is a consequence of another insightful work of Thom, \textit{cfr.} \cite{Thom69}, where he studied manifolds with singularities partitioning them into partially ordered \emph{strata} of varying dimensions; each such stratum has a neighborhood which is a locally trivial bundle with fiber the cone on a compact manifold with singularities, whose partially ordered set of strata has smaller dimension. This gives rise to a recursive construction that enabled Thom to understand and provide a geometric description of singularities.

We will now exploit the geometric obstruction theory described in \cite{SullivanLiverpool} for reducing the dimension of singularites of a cycle. In particular, suppose $T$ is a triangulated space of dimension $m$, and $\mathcal{S}_T$ its \emph{singularity locus} of dimension $s$. Then, for every $s$-dimensional simplex of $\mathcal{S}_T$ we can consider its \emph{link}, which is a well-defined $(m-s-1)$-dimensional manifold; this link determines an element in a suitable cobordism group $\Omega$ and the sum of the singular simplices with these link coefficients forms a cycle which defines an obstruction, that is $$\omega_T \in \mathbf{H}_s(\mathcal{S}_T, \Omega).$$ If this obstruction vanishes, then it is possible to resolve the singularity by a blow-up technique and reduce their dimension, \textit{cfr.} \cite[Theorem D]{SullivanLiverpool}. Geometrically, this means that any singular cycle representing a homology class can be resolved by replacing each conic fiber of the top singularity stratum by compact manifolds bounding the links, provided each link bounds a compact submanifold; the recurrence stops as soon as a link of singularities which is not null-cobordant is met.

In particular, if $T$ is an $m$-dimensional oriented geometric cycle, the natural obstructions lie in $$\mathbf{H}_s(T,\widetilde{\Omega}_r),$$ where $\widetilde{\Omega}_r$ denotes the $r$-dimensional oriented cobordism group and $r=m-s-1$, which coincides with the dimension of the link of each $s$-dimensional simplex of $\mathcal{S}_T$; we refer to \cite[\S 17]{MilnorStasheff} for an introduction about the oriented cobordism graded ring $\widetilde{\Omega}_*$.

\begin{thm}\label{t:Thom_innatelysingular}
    Let $z \in \mathbf{H}_7(\mathcal{M},\mathbb{Z})$ be the Thom homology class of Example \ref{ex:ThomLensSpaces} and fix a triangulation $\mathcal{K}$ of the smooth oriented closed manifold $\mathcal{M}$. Then it is impossible to find a representative $\Sigma$ for $z$ which is a smooth embedded submanifold in the complement of the $1$-dimensional skeleton of $\mathcal{K}$.
\end{thm}

\begin{proof}
    \hspace{-0.27cm}\footnote{We are grateful to Dennis Sullivan for this elegant proof and enjoyable conversations.}Denote by $T$ the 7-dimensional cycle representing Thom's homology class, and consider its singularity locus $\mathcal{S}_T$. Towards a proof by contradiction, assume that the cycle is a substratified set of a Whitney stratification of $\mathcal{M}$ which only intersects the one-skeleton $\mathcal{K}^1$ of a triangulation compatible with the stratification\footnote{We refer to \cite{Thom69} and \cite{GoreskyMacPherson} for the notions about stratification theory.}.

    For each 1-dimensional simplex in the cycle we consider its link, which is a 5-dimensional closed oriented manifold. Since the obstruction to the resolution of singularities is an element of $$\omega_{T} \in \mathbf{H}_1(\mathcal{S}_T, \widetilde{\Omega}_5)$$ and the 5-dimensional oriented cobordism group $\widetilde{\Omega}_5\simeq \mathbb{Z}_2$, by \cite[Theorem D]{SullivanLiverpool} the cycle $2T$ can be resolved to the lower dimensional stratum, \textit{i.e.} the zero-skeleton $\mathcal{K}^0$.

    Analogously, the link of each vertex is a $6$-dimensional closed oriented manifold and the oriented cobordism group $\widetilde{\Omega}_6$ is trivial; thus, there is no obstruction to a full resolution of the singularities of $2T$, and hence of $2z$. This is clearly in contradiction with Thom's algebraic obstruction which is 3-torsion, and that cannot be resolved if we multiply Thom's homology class $z$ by a factor 2; that is $St_3^5(2y)\neq 0,$ where $y$ is the Poincaré dual of $z$.
\end{proof}

%\begin{remark}
%Let $z \in \mathbf{H}_7(\mathcal{M},\mathbb{Z})$ denote Thom's homology class of Example \ref{ex:ThomLensSpaces} and call $y$ its Poincaré dual. By Theorem \ref{t:Thom_innatelysingular} we conclude that there exists a $\delta^*$-neighborhood $U_{\delta^*}$ of the 1-skeleton $\mathcal{K}^1$ of a triangulation $\mathcal{K}$ of $\mathcal{M}$ such that the restriction of $y$ to $\mathcal{M}\setminus U_{\delta^*}$ cannot be represented by a smooth submanifold; this shows the optimality of Theorem \ref{t:1}.
%\end{remark}

\begin{remark}
    We remark that in \cite[page 20]{AlmgrenBrowder88-91} the counterexample to the construction is not correct, since by \cite[Corollaire II.28]{Thom54} every 5-dimensional integral homology class in an oriented closed smooth manifold is representable by a smooth embedded submanifold, and hence part $(3)$ of Theorem \ref{t:1} provides the desired smooth approximation.
\end{remark}

\begin{remark}
    As a byproduct of the proof of Lemma \ref{l:Phi}, it is also possible to show the following. Let $\mathcal{M}, \tau$ and $T$ as in Assumption \ref{a:1} and denote $\operatorname{Sing}(T)$ its singular set (in the sense of \cite[Definition  0.2]{dls3}). If $\mathcal{H}^k(\operatorname{Sing}(T))=0$ for any $k \in \{1,\dots,m\}$, then for every triangulation $\mathcal{K}$ of $\mathcal{M}$ there exists an integral current $T'$ homologous to $T$ such that $\operatorname{Sing}(T')\subset \mathcal{K}^{k-1}$ with $T'$ smooth in $\mathcal{M}\setminus \mathcal{K}^{k-1}$. Theorem \ref{t:Thom_innatelysingular} implies that, in general, for an integral current $T$ representing an integral homology class $\tau \in \mathbf{H}_m(\mathcal{M},\mathbb{Z})$ it is not possible to conclude that $\mathcal{H}^{m-5}(\operatorname{Sing}(T))=0$.
\end{remark}

\newpage
\addtocontents{toc}{\protect\setcounter{tocdepth}{1}}

\appendix

\section{Useful lemmas on triangulations and simplicial decompositions}\label{s:Appendix_triangulations}

In this section we collect a few elementary facts about triangulating regions which are used in the paper. 

\subsection{Algorithm to subdivide a convex polytope}\label{s:polytope-subdivision} First of all, we understand a convex polytope $P$ of $\mathbb R^N$ as a closed convex set with a finite number of extremal points. Given a convex polytope $P\subset \mathbb R^N$ we now describe an algorithm to triangulate it. For any convex polytope we define its barycenter as the point which is given by the convex combination of the extremal points with all equal weights (if $V_1, \ldots, V_m$ are the extremal points of $P$, then the baricenter is $\frac{1}{m} \sum_i V_i)$.

First of all, a touching hyperplane $\pi$ of $P$ is an hyperplane such that
\begin{itemize}
\item $\pi\cap P$ is nonempty;
\item $P$ is contained in one of the two closed half-spaces bounded by $\pi$.
\end{itemize}
If the dimension of $P$ is strictly smaller than $N$, then we let the set $\mathcal{F}$ of faces of $P$ be the collection of convex subsets of $P$ of the form $P\cap \pi$, where $\pi$ varies among all touching hyperplanes. If the dimension of $P$ is $N$ we add to $\mathcal{F}$ the polytope $P$ itself. We subdivide $\mathcal{F}$ as 
\[
\bigcup_{k=0}^{{\rm dim}\, (P)} \mathcal{F}_k\, ,
\]
where $\mathcal{F}_k = \{F\in \mathcal{F}: {\rm dim}\, (F)=k\}$. 
Clearly $\mathcal{F}_0$ consists of points and it is the set of extremal points of $P$ (in particular, it is a finite set) and any other element of $\mathcal{F}$ is necessarily the convex hull of some appropriate subset of $\mathcal{F}_0$. 

In order to triangulate $P$, we first observe that all $1$-dimensional faces and all $0$-dimensional faces are by definition simplices of the corresponding dimension. We then list the $2$-dimensional faces. For each face $F$ which is not a triangle we consider the barycenter $b(F)$ and we decompose $F$ into the triangles formed by $b(F)$ and the sides of $F$ (namely, the $1$-dimensional faces of $F$). Note that the collection of all such triangles (as $F$ also varies among all 2-dimensional faces) has the following property: the intersection of any pair of such triangles is either empty, or a common vertex, or a common side. Next, fix a $3$-dimensional face $F$ which is not a simplex. Each of its $2$-dimensional faces $G$ is decomposed in triangles $T$'s in the previous step. Decompose $F$ in the $3$-dimensional simplices constructed as convex hulls of any such $T$ and the barycenter $b(F)$ of $F$ (we can think of them as pyramids with basis $T$ and vertex $b(F)$). We have decomposed all 3-dimensional faces into $3$-dimensional simplices $S$. Any pair of such $S$ (irrespectively of whether they belong to the decomposition of the same 3-dimensional face or to the decompositions of two distinct faces) has the property that their intersection is a common lower-dimensional face. We proceed inductively increasing the dimension of the faces at each step until we reach (and include) the one of highest dimension, namely $P$.  

The following elementary lemmas will play an important role.

\begin{lem}\label{l:triang-1}
If $P\subset \mathbb R^N$ is a convex polytope and $A: \mathbb R^N\to \mathbb R^N$ an affine invertible map, then the triangulation $\mathcal{T}'$ for $A(P)$ obtained through the algorithm above coincides with the image through $A$ of the triangulation $\mathcal{T}$ obtained for $P$  through the algorithm, namely $\mathcal{T}' = \{A(T): T\in \mathcal{T}\}$. 
\end{lem}

\begin{lem}\label{l:triang-2}
Let $P, P'\subset \mathbb R^N$ be two convex polytopes whose intersection is a common face of both. Consider the triangulation $\mathcal{T}$ of $P$ and the triangulation $\mathcal{T}'$ of $P'$ obtained applying the algorithm above. Then the union of the two triangulations is a triangulation, namely: the intersection of an arbitrary element of $\mathcal{T}$ with an arbitrary element of $\mathcal{T}'$ is a common face of both simplices.
\end{lem}

\subsection{Embedding convex polytopes into skeleta of refined triangulations}
\label{s:embedding-1}

In this section we prove the following proposition.

\begin{pro}\label{p:nello-scheletro}
Consider a finite family of convex $m$-dimensional polytopes $\{P_i\}$ in $\mathbb R^N$ and a triangulation $\mathcal{T}$ of some closed subset of $\mathbb R^N$. Then there exists a triangulation $\mathcal{T}_f$ finer than $\mathcal{T}$ with the property that each polytope $P_i$ is union of elements of the $m$-skeleton of $\mathcal{T}_f$. Moreover, the refinement of $\mathcal{T}$ is local in the following sense: if we denote by $\mathcal{T}'$ the collection of those $N$-dimensional simplices of $\mathcal{T}$ which intersect at least one $P_i$, then any simplex of $\mathcal{T}$ which does not intersect an element of $\mathcal{T}'$ is not refined (namely, it is also an element of $\mathcal{T}_f$).
\end{pro}

We will give an algorithm to produce $\mathcal{T}_f$. First of all, take all possible intersections of the $P_i$'s with $N$-dimensional simplices of $\mathcal{T}$. This gives a collection of new $m$-dimensional polytopes $\bar P_j$: each of them is contained in an $N$-dimensional simplex of $\mathcal{T}$, which we denote by $T_j$.

We start with $\bar P_1$. Because it is a convex polytope of dimension $<N$, we write it as the intersection of a finite number of hyperplanes $\pi_i$ and a finite number of closed halfspaces $H_j$. Then we build a finite collection $\mathcal{H}$ of pairs of halfspaces $H_i^-$, $H_i^+$ by adding for each $H_j^+:= H_j$ the closure of its complement $H_j^-$, and for ever $\pi_i$ the pair of closed halfspaces which have $\pi_i$ as a boundary. We then subdivide $T_1$ inductively in smaller convex $N$-dimensional polytopes in the following way. In the first step we keep $T_1$ if it is contained in one of the two halfspaces $\{H_1^+, H_1^-\}$, otherwise we replace it with the pair $\{H_1^+\cap T_1, H_1^-\cap T_1\}$. At step $j$ we assume to have a finite collection of closed convex $N$-dimensional polytopes and each of them is kept if it is contained in one of the two halfspaces $H_{j+1}^+, H_{j+1}^-$, otherwise it is replaced by the two intersections with them.

The resulting collection is a partition of $T_1$ into convex polytopes with the property that any two faces of any two polytopes intersect in a common face. Moreover, the original $P_1$ is the $m$-dimensional face of some convex polytope of this partition. We apply to each of these $N$-dimensional polytopes the triangulating algorithm of Section \ref{s:polytope-subdivision} and, by Lemma \ref{l:triang-1}, we obtain a triangulation of $T_1$. However, $T_1$ has faces in common with other simplices of $\mathcal{T}$ which are not yet partitioned. In order to remedy, we proceed inductively as follows. We first denote by $\mathcal{S}_k$ be the collection of $k$-dimensional faces of $T_1$; we start with the edges $\mathcal{S}_1$ and add to $\mathcal{S}_2$ every triangle of $\mathcal{T}$ which contains an edge $\sigma\in \mathcal{S}_1$ and add it to $\mathcal{S}_2$. The new triangles are triangulated compatibly with the elements of $\mathcal{S}_1$ by adding its barycenter and connecting it with edges to all its vertices and all the new points in the edges of $\mathcal{S}_1$ that it might contain. Observe that this procedure does not subdivide any edge of the initial triangulation $\mathcal{T}$ which is not in $\mathcal{S}_1$. Similarly, at step $j$ we enlarge $\mathcal{S}_{j+1}$ with all the $j+1$-dimensional simplices which contain an element of $\mathcal{S}_j$ as a face. Each new simplex $S$ added is triangulated by considering its barycenter $b (S)$ and subdividing $S$ into the pyramids which:
\begin{itemize}
    \item have vertex $b(S)$ and basis a $j$-dimensional face $F$ of $S$, in case $F$ does not belong to $\mathcal{S}_j$;
    \item have vertex $b(S)$ and basis a $j$-dimensional simplex of the subdivision of the face $F\in \mathcal{S}_j$ obtained so far inductively if the other case;
\end{itemize}
We stop the procedure when we have subdivided the final new elements added to the collection $\mathcal{S}_N$. The final result is a triangulation.

Observe that, by construction, in the new triangulation $\mathcal{T}_1$ the polytope $P_1$ is the union of elements of the $m$-dimensional skeleton.

We now proceed inductively with the subsequent polytopes. However, observe that, at the $j+1$-th step, the simplex $T_{j+1}$ might not be an element of the triangulation $\mathcal{T}_j$. However, if that is the case, by construction there is a collection $\mathcal{C}$ of $N$-dimensional simplices of $\mathcal{T}_j$ whose union is precisely $T_{j+1}$. The first subdivision algorithm in which we intersected $T_1$ with halfspaces is in this case applied simultaneously to all of the elements of $\mathcal{C}$. This then results into a subdivision $\mathcal{S}$ of $T_{j+1}$ into convex polytopes which has the two properties of the subdivision obtained in the previous argument for $T_1$. This subdivision has however the additional feature that, for any element $C$ of $\mathcal{C}$, there is an appropriate subcollection $\mathcal{S}'$ of $\mathcal{S}$ which is in fact a subdivision of it. The remaining part of the algorithm outlined above is then applied verbatim and the result is the next triangulation $\mathcal{T}_{j+1}$. 

\subsection{Embedding polytopes in skeleta of refinements of polyhedra}\label{s:embedding-2}

In this section we extend the algorithm of the previous subsection to handle more general piecewise linear ambient closed manifolds. For simplicity we assume that the latter are suitably embedded into some higher-dimensional Euclidean space. 

\begin{definition}\label{d:polyhedra}
A finite polyhedron in $\mathbb R^N$ is the collection of finitely many simplices of $\mathbb R^N$.
\end{definition}

A finite poyhedron $K$ always admits a finite triangulation, namely a finite collection of simplices $\mathcal{T}$ with the following properties:
\begin{itemize}
\item Any face of an element of $\mathcal{T}$ belongs to $\mathcal{T}$;
\item The intersection of any two elements of $\mathcal{T}$ is always either empty or a face of both;
\item The union of the elements of $\mathcal{T}$ is $K$.
\end{itemize}
Although this is a classical fact, note that it is also a consequence of Proposition \ref{p:nello-scheletro}.

\begin{definition}\label{d:PL-manifolds}
We will consider continuous maps $f$ over finite polyhedra $K$ taking values into a smooth manifold $\mathcal{M}$. Such maps $f$ will be called piecewise smooth if there is a triangulation $\mathcal{T}$ of $K$ with the property that the restriction of $f$ to every simplex in $\mathcal{T}$ is smooth. The map will be called a piecewise smooth homeomorphism if in addition it is an homeomorphism with the image and if the triangulation can be chosen so that, for every simplex $\sigma\in \mathcal{T}$, the differential $D (f|_\sigma)$ of the restriction of $f$ to $\sigma$ has maximal rank at every point. When such a map exists between some polyhedron $K$ and some smooth compact manifold $\mathcal{M}$ without boundary, we say that $K$ is an embedded piecewise linear closed submanifold of $\mathbb R^N$.
\end{definition}

It is a classical result of Whitehead that if $f: K \to \mathcal{M}$ is a piecewise smooth homeomorphism, then every pair of triangulations of $K$ admits a further triangulation which is a common refinement of both. 

The generalization of Proposition \ref{p:nello-scheletro} that we are looking for
is then the following.

\begin{pro}\label{p:nello-scheletro-2}
Consider a polyhedron $K$ which is a piecewise linear $m+n$-dimensional submanifold of $\mathbb R^N$, let $\{P_i\}$ be a finite collection of $m$-dimensional convex polytopes all contained in $K$, and let $\mathcal{T}$ be a triangulation of $K$. Then there is a triangulation $\mathcal{T}_f$ of $K$ which refines $\mathcal{T}$ and has the property that every $P_i$ is the union of finitely many elements of the $m$-skeleton of $\mathcal{T}_f$. 
\end{pro}

We quickly describe how to modify the algorithm explained in Section \ref{s:embedding-1}. As in there, we intersect the polytopes with the $m$-dimensional simplices of $\mathcal{T}$, reducing the proposition to the case in which each $P_i$ is contained in an $m+n$-dimensional simplex of $T_i$. Moreover, as in there, we refine $\mathcal{T}_0=\mathcal{T}$ into $\mathcal{T}_1$, $\mathcal{T}_2$, and so on, ``embedding'' one $P_i$ at a time. 

At the starting step the algorithm gives first a way to triangulate $T_1$ so that $P_1$ is the union of the $m$-skeleton of this local triangulation. In the argument, $T_1$ is supposed to be an $m+n$-dimensional simplex of $\mathbb R^{m+n}$, but this can be easily achieved identifying the $m+n$-dimensional affine plane $\pi$ containing $T_1$ with $\mathbb R^{m+n}$. We then further triangulate all simplices of $\mathcal{T}$ which intersect $T_1$, proceeding inductively from the lower dimensional ones. Since at each stage of this second algorithm a single simplex is considered at a time, we can think of this as also taking place in some Euclidean space. 

At the inductive step, when embedding $P_{j+1}$ into a refinement of $\mathcal{T}_j$, the only difference is that the first subdivision is carried over all at once on all the $m+n$-dimensional simplices $\mathcal{C}_j$ of $\mathcal{T}_j$ which are contained in $T_{j+1}$. Again, the only important point is that, like above, $T_{j+1}$ is an $m+n$-dimensional simplex. The second part, which refines the triangulation of $\mathcal{T}_j$ over all simplices intersecting at least one element of $\mathcal{C}_j$, is the same as in the initial step. 

\section{Cohomology operations and characteristic classes}\label{a:cohomology}
\label{s:cohomologyoperations}

In this section we collect a few results about cohomology operations, characteristic classes and the cohomology of $\BSOn$, see also \cite{MosherTangora, Spanier, MilnorStasheff, DubrovinFomenkoNovikovIII}.

\subsection{Cohomology operations}

A \emph{cohomology operation} of type $(\pi,n;\rho, m)$ is a family of functions $$\theta_X:\mathbf{H}^n(X,\pi) \rightarrow \mathbf{H}^m(X,\rho),$$ one for each space $X$, satisfying the \emph{naturality} condition $f^*\theta_Y=\theta_X\,f^*$ for any map $f:X \rightarrow Y$. The set of cohomology operations of type $(\pi,n;\rho, m)$ can be denoted by $\mathcal{O}(\pi,n;\rho, m)$. A cohomology operation $\theta$ is said to be \emph{additive} if $\theta_X$ is a homomorphism for every $X$. An important result on the classification of these operations in terms of the cohomology of Eilenberg-MacLane spaces is the following.

\begin{thm}
    There is a one-to-one correspondence $$\mathcal{O}(\pi,n;\rho, m) \rightarrow	\mathbf{H}^m(\mathbf{K}(\pi, n), \rho),$$ given by $\theta \rightarrow \theta(\iota_n)$, where $\iota_n$ is the fundamental class of $\mathbf{K}(\pi, n)$.
\end{thm}

%Of specific interest is the following particular case: $\pi=\rho=\mathbb{Z}_p$ and $m=n+k$. Hence, we can write $$[\mathbf{K}(\mathbb{Z}_p, n), \mathbf{K}(\mathbb{Z}_p, n+k)] \simeq \mathbf{H}^{n+k}(\mathbf{K}(\mathbb{Z}_p, n), \mathbb{Z}_p),$$ so that composing with an element of this cohomology group gives a natural transformation (a cohomology operation) from $\mathbf{H}^n(X , \mathbb{Z}_p)$ to $\mathbf{H}^{n+k}(X,\mathbb{Z}_p)$.

%If $\theta \in \mathcal{O}(n, m , \pi, \rho)$ and $\phi \in \mathcal{O}(m, l , \rho, \tau)$, then the cohomology operation $\phi \circ \theta \in \mathcal{O}(n, l , \pi, \tau)$ is well-defined. In particular, one can define the composite $\phi \circ \theta \in \mathcal{O}_S(\pi, \tau)$ of any two stable cohomology operations $\theta \in \mathcal{O}_S(\pi, \rho)$ and $\phi \in \mathcal{O}_S(\rho, \tau)$, so that the group $\mathcal{O}_S(\pi, \pi)$ is a ring: composition endows the set of stable operations with a natural ring structure. In particular, $\mathcal{O}_S\left(\mathbf{Z}_p, \mathbf{Z}_p\right)$ is called the \emph{Steenrod algebra} and denoted $\mathcal{A}_p$. 

The \emph{Steenrod squares} $S q^i$, $i\geq 0$, are additive cohomology operations of type $(\mathbb{Z}_2,n;\mathbb{Z}_2,n+i)$, $$S q^i: \mathbf{H}^n(X , \mathbb{Z}_2) \rightarrow \mathbf{H}^{n+i}(X , \mathbb{Z}_2),$$ defined for all $n$ and such that
\begin{enumerate}
\item $S q^0=1$, the identity;
\item if $\text{deg }u=n$, then $Sq^nu=u \smile u$;
\item if $i> \text{deg }u$, then $Sq^iu=0$;
\item if $u,v \in \mathbf{H}^*(X, \mathbb{Z}_2)$, then $$Sq^k(u\smile v)= \sum_{i+j=k}Sq^iu \smile Sq^{\,j}v.$$ This condition is usually called \emph{Cartan formula}.
\end{enumerate}
The above properties characterize the cohomology operations $Sq^i$. It is then possible to prove existence and uniqueness of such operations, \textit{cfr.} \cite{SteenrodEpstein}. From the above properties it is possible to derive the following.
\begin{enumerate}\setcounter{enumi}{4}
\item $S q^1$ is the Bockstein homomorphism $\beta$ induced by the short exact coefficient sequence $$0 \rightarrow \mathbb{Z}_2 \rightarrow \mathbb{Z}_4 \rightarrow \mathbb{Z}_2\rightarrow 0;$$
\item if $0 < a < 2b$, then $$Sq^aSq^b=\sum^{\lfloor a/2 \rfloor}_{j=0} \binom{b-1-j}{a-2j}Sq^{a+b-j}Sq^{\,j},$$ where the binomial coefficient is taken mod 2. These relations are usually called \emph{Adem relations};
\item $Sq^i(\sigma(u))= \sigma(Sq^iu)$, where $\sigma: \mathbf{H}^n(X,\mathbb{Z}_2) \rightarrow \mathbf{H}^{n+1}(\Sigma X,\mathbb{Z}_2)$ is the suspension isomorphism given by reduced cross-product with a generator of $\mathbf{H}^1(S^1,\mathbb{Z}_2)$ and $\Sigma X$ the reduced suspension of $X$.
\end{enumerate}
This last property says that the Steenrod squares are \emph{stable} operations, \textit{i.e.} they commute with the cohomology suspension operation. %We recall that the cup product square operation $u \mapsto u\smile u$ is not stable.
\medskip

There are analogous additive operations for odd primary coefficients: the \emph{reduced Steenrod $p^{\text{th}}$-powers} $\mathcal{P}^i$ of type $\mathcal{O}(\mathbb{Z}_p, n; \mathbb{Z}_p, n+2i(p-1))$ for $p$ an odd prime\footnote{Sometimes the notation $\mathcal{P}^i_p$ is used to highlight the coefficient group $\mathbb{Z}_p$.} and written $$\mathcal{P}^i: \mathbf{H}^n(X, \mathbb{Z}_p) \rightarrow \mathbf{H}^{n+ 2i(p-1)}(X, \mathbb{Z}_p).$$ 
They satisfy the following properties.
\begin{enumerate}
\item $\mathcal{P}^0=1$, the identity;
\item if $\text{deg }u=2i$, then $\mathcal{P}^iu=u \smile \cdots \smile u$, $p$ times;
\item if $2i> \text{deg }u$, then $\mathcal{P}^iu=0$;
\item if $u,v \in \mathbf{H}^*(X, \mathbb{Z}_p)$, then $$\mathcal{P}^k(u\smile v)= \sum_{i+j=k}\mathcal{P}^iu \smile \mathcal{P}^{\,j}v.$$ This is the \emph{Cartan formula};
\end{enumerate}

In analogy with Steenrod squares, the above properties characterize the cohomology operations $\mathcal{P}^i$: it is then possible to prove existence and uniqueness of such operations, \textit{cfr.} \cite{SteenrodEpstein}. From the above properties it is possible to derive the following.

\begin{enumerate}\setcounter{enumi}{4}
\item $\mathcal{P}^i(\sigma(u))= \sigma(\mathcal{P}^iu)$, where $\sigma: \mathbf{H}^n(X,\mathbb{Z}_p) \rightarrow \mathbf{H}^{n+1}(\Sigma X,\mathbb{Z}_p)$ is the suspension isomorphism given by reduced cross-product with a generator of $\mathbf{H}^1(S^1,\mathbb{Z}_p)$;
\item if $a < pb$, then $$\mathcal{P}^a\mathcal{P}^b=\sum^{\lfloor a/p \rfloor}_{j=0} (-1)^{a+j}\binom{(p-1)(b-j)-1}{a-pj}\mathcal{P}^{a+b-j}\mathcal{P}^{\,j}.$$ These are the \emph{Adem relations}.
\end{enumerate}
We note that by the Adem relations, the operation $Sq^{2i+1}$ is the same as the composition $Sq^1Sq^{2i}=\beta Sq^{2i}$, so that $Sq^{2i}$ can be understood as $\mathcal{P}^i$ for $p=2$.

Another additive cohomology operation for odd primary coefficients is the \emph{Bockstein homomorphism} $\beta_p$ of type $\mathcal{O}(\mathbb{Z}_p, n; \mathbb{Z}_p, n+1)$ for $p$ an odd prime and written $$\beta_p: \mathbf{H}^n(X, \mathbb{Z}_p) \rightarrow \mathbf{H}^{n+1}(X, \mathbb{Z}_p),$$ which is obtained from the short exact coefficient sequence $0 \rightarrow \mathbb{Z}_p \rightarrow \mathbb{Z}_{p^2} \rightarrow \mathbb{Z}_p \rightarrow 0.$

Composition endows the set of stable cohomology operations with a natural ring structure: this ring is known as the \emph{Steenrod algebra} and usually denoted by $\mathcal{A}_p$. The Steenrod algebra $\mathcal{A}_2$ is defined to be the the algebra over $\mathbb{Z}_2$ that is the quotient of the algebra of polynomials in the noncommuting variables $Sq^i$, $i\geq 1$, by the two-sided ideal generated by the Adem relations. Analogously, the Steenrod algebra $\mathcal{A}_p$ for odd $p$ is defined to be the algebra over $\mathbb{Z}_p$ formed by polynomials in the noncommuting variables $\beta_p, \mathcal{P}^i$, $i\geq 1$, modulo the Adem relations and the relations $\beta_p^2=0$. Thus, for every space $X$, $\mathbf{H}^*(X,\mathbb{Z}_p)$ is a module over $\mathcal{A}_p$ for all primes $p$; the Steenrod algebra is a graded algebra with the elements of degree $i$ being those that map $\mathbf{H}^n(X,\mathbb{Z}_p)$ to $\mathbf{H}^{n+i}(X,\mathbb{Z}_p)$ for all $n$. It is possible to prove that $\mathcal{A}_2$ is generated as an algebra by the elements $Sq^{2^{k}}$ and $\mathcal{A}_p$ for $p$ odd prime is generated by $\beta_p$ and the elements $\mathcal{P}^{p^k}$.

More generally, let $R$ be a commutative ring with unit. We recall that on the category of free chain complexes $C$ over $R$ and short exact sequences of $R$ modules \begin{equation}\label{e:BocksteinC}0\rightarrow G' \xrightarrow{\varphi} G \xrightarrow{\psi} G''\rightarrow 0\end{equation} there is a functorial connecting homomorphism $$\beta^*:  \mathbf{H}^*(C, G'') \rightarrow \mathbf{H}^*(C, G')$$ of degree 1 and a functorial exact sequence $$\dots \rightarrow \mathbf{H}^n(C, G') \xrightarrow{\varphi^*} \mathbf{H}^n(C, G) \xrightarrow{\psi^*} \mathbf{H}^n(C, G'') \xrightarrow{\beta^*} \mathbf{H}^{n+1}(C, G')\rightarrow \dots,$$ \textit{cfr.} \cite[Theorem 4.5.11]{Spanier}. The connecting homomorphism $\beta^*$ (sometimes just denoted $\beta$ when the coefficient group is clear from the context) is called the Bockstein cohomology homomorphism corresponding to the coefficient sequence \eqref{e:BocksteinC}.

\subsection{Characteristic classes}

We define \emph{Stiefel-Whitney cohomology classes} of a vector bundle axiomatically. For a proof of existence and uniqueness of cohomology classes satisfying these 4 axioms we refer to \cite{MilnorStasheff}. 

\begin{enumerate}
\item To each real vector bundle $\xi$ with base space $B(\xi)$ there corresponds a sequence of cohomology classes $$
w_{i}(\xi) \in \mathbf{H}^{i}(B(\xi) , \mathbb{Z}_2), \quad i=0,1,2, \ldots,$$ called the \emph{Stiefel-Whitney classes} of $\xi$. The class $w_0(\xi)$ is equal to the unit element $$1 \in \mathbf{H}^0(B(\xi),\mathbb{Z}_2)$$ and $w_{i}(\xi)$ equals zero for $i>n$ if $\xi$ is an $n$-plane bundle.

\item (Naturality) If $f: B(\xi) \rightarrow B(\eta)$ is covered by a bundle map from $\xi$ to $\eta$, then $$w_{i}(\xi)=f^* w_{i}(\eta).$$

\item (Whitney Product Theorem) If $\xi$ and $\eta$ are vector bundles over the same base space, then the Stiefel-Whitney class of a direct sum is the cup product of the summands' classes $$w_{k}(\xi \oplus \eta)=\sum_{i=0}^{k} w_{i}(\xi) \smile w_{k-i}(\eta).$$
%For example $w_1(\xi \oplus \eta)=w_1(\xi)+w_1(\eta)$, $w_2(\xi \oplus \eta)=w_2(\xi)+w_1(\xi) w_1(\eta)+w_2(\eta),$ and so on (omitting the cup product in the notation).

\item For the canonical line bundle\footnote{Let $E(\gamma_n^1)$ be the subset of $\mathbb{R}\mathbb{P}(n) \times \mathbb{R}^{n+1}$ consisting of all pairs $(\{\pm x\},v)$ such that the vector $v$ is a multiple of $x$. Recall that the \emph{canonical line bundle} over the real projective space $\mathbb{R}\mathbb{P}(n)$ is the vector bundle $\pi: E(\gamma_n^1) \rightarrow \mathbb{R}\mathbb{P}(n)$ defined by $\pi(\{\pm x\},v)=\{\pm x\}$. Thus each fiber $\pi^{-1}(\{\pm x\})$ can be identified with the line through $x$ and $-x$ in $\mathbb{R}^{n+1}$; each such line is to be given the usual vector space structure.} $\gamma_1^1$ over $\mathbb{R}\mathbb{P}(1)$, the Stiefel-Whitney class $w_1(\gamma_1^1)$ is non-zero.
\end{enumerate}

\begin{pro}\label{p:stiefelsteenrod} Stiefel-Whitney classes $w_i(\xi) \in \mathbf{H}^i(B)$ can be characterized in terms of the Steenrod operations by showing the following equality: $$w_i(\xi)=\Phi^{-1}Sq^i\Phi(1)=\Phi^{-1}Sq^iu,$$ where $\Phi$ is the Thom isomorphism and $u \in \widetilde{\mathbf{H}}^n(T (\xi), \mathbb{Z}_2)$ is the Thom class of $\xi$.\end{pro} This shows that $w_i(\xi)$ is the unique cohomology class in $\mathbf{H}^i(B)$ such that $\Phi(w_i(\xi))= \pi^*w_i(\xi)\smile u$ is equal to $Sq^i\Phi(1)=Sq^iu$, \textit{cfr.} also \cite{Thomdottorato}.

All the discussion about Stiefel-Whitney classes works analogously for complex vector bundles, except that for complex vector bundles all the cohomology classes belong to $\mathbb{Z}$ coefficient cohomology: they are called the \emph{Chern classes}. One possible way to define Chern classes is the following.

There is a unique sequence of functions $c_0, c_1, c_2, \cdots$ assigning to each complex vector bundle $\omega$ with $E \xrightarrow{\pi} B$ a class $c_i(\omega) \in \mathbf{H}^{2 i}(B , \mathbb{Z})$, depending only on the isomorphism type of $\omega$, such that:
\begin{enumerate}
\item $c_0(\omega)$ equals the unit element $1 \in \mathbf{H}^0(B,\mathbb{Z})$ and $c_i(\omega)=0$ for $i>n$ if $\omega$ a complex $n$-plane bundle;
\item $c_i(f^*(\omega))=f^*(c_i(\omega))$, for a pull-back bundle $f^*(\omega)$;
\item if $\omega_1$ is a complex $n$-plane bundle and $\omega_2$ a complex $m$-plane bundle, then $$c_k(\omega_1 \oplus \omega_2)=\sum_{i=0}^k c_i(\omega_1) \smile c_{k-i}(\omega_2);$$
\item for the canonical line bundle $\omega$ with $E \xrightarrow{\pi} \mathbb{CP}(1)$, $c_1(E)$ is a generator of $\mathbf{H}^2(\mathbb{CP}(1), \mathbb{Z})$ specified in advance.
\end{enumerate}
%Besides the evident formal similarity between Stiefel-Whitney and Chern classes there is also the following direct relation: regarding an $n$-dimensional complex vector bundle $\omega$ as a $2 n$-dimensional real vector bundle, then $w_{2 i+1}(\omega)=0$ and $w_{2 i}(\omega)$ is the image of $c_i(\omega)$ under the coefficient homomorphism $\mathbf{H}^{2 i}(B(\omega) , \mathbb{Z}) \rightarrow \mathbf{H}^{2 i}(B(\omega) , \mathbb{Z}_2)$.

We now define the \emph{Pontryagin classes} $p_i(\xi) \in \mathbf{H}^{4 i}(B , \mathbb{Z})$ associated to a $n$-plane bundle $\xi$ in terms of Chern classes. For an $n$-plane bundle $\xi$ with $E \rightarrow B$, its complexification is the complex $n$-plane bundle $\xi^{\mathbb{C}}$ with $E^{\mathbb{C}} \rightarrow B$ obtained from the real $n$-plane bundle $\xi \oplus \xi$ by defining scalar multiplication by the complex number $i$ in each fiber $\mathbb{R}^n \oplus \mathbb{R}^n$ via the rule $i(x, y)=(-y, x)$. Thus, each fiber $\mathbb{R}^n$ of $\xi$ becomes a fiber $\mathbb{C}^n$ of $\xi^{\mathbb{C}}$. The Pontryagin class $p_i(\xi)$ is then defined to be $$p_i(\xi):=(-1)^i c_{2 i}(\xi^{\mathbb{C}}) \in \mathbf{H}^{4 i}(B , \mathbb{Z}).$$

Let $\xi$ be an oriented (real) $n$-plane bundle $E \xrightarrow{\pi} B$ and consider the restriction homomorphism $\widetilde{\mathbf{H}}^*(T (\xi), \mathbb{Z}) \rightarrow \mathbf{H}^*(E, \mathbb{Z})$ induced by the inclusion and denoted as $y \mapsto y_{|E}$. In particular, applying this homomorphism to the Thom class $u \in \widetilde{\mathbf{H}}^n(T (\xi), \mathbb{Z})$, we obtain a new cohomology class $$u_{|E} \in \mathbf{H}^n(E,\mathbb{Z}).$$ Recalling that $\mathbf{H}^n(E,\mathbb{Z})$ is canonically isomorphic to $\mathbf{H}^n(B,\mathbb{Z})$, we can define the \emph{Euler class} of an $n$-plane bundle $\xi$ as the cohomology class $$e(\xi)\in \mathbf{H}^n(B,\mathbb{Z})$$ corresponding to $u_{|E}$ under the isomorphism $\pi^*:\mathbf{H}^n(B,\mathbb{Z}) \rightarrow \mathbf{H}^n(E,\mathbb{Z})$.

%Here we recall some basic properties of Euler classes $e(\xi) \in \mathbf{H}^n(B , \mathbb{Z})$ associated to oriented $n$-plane bundles $\xi$.

%\begin{enumerate} 
%\item An orientation of a vector bundle $\xi$ with $E \xrightarrow{\pi} B$ induces an orientation of a pull-back bundle $f^*(\xi)$ such that $e(f^*(\xi))=f^*(e(\xi))$.
%\item Orientations of vector bundles $\xi_1$ and $\xi_2$ determine an orientation of the direct sum $\xi_1 \oplus \xi_2$ such that $e(\xi_1 \oplus \xi_2)=e(\xi_1) \smile e(\xi_2)$.
%\item For an orientable $n$-plane bundle $\xi$, the coefficient homomorphism $\mathbf{H}^n(B , \mathbb{Z}) \rightarrow \mathbf{H}^n(B ,\mathbb{Z}_2)$ carries $e(\xi)$ to $w_n(\xi)$. For a complex $n$-plane bundle $\omega$ there is the relation $e(\omega)=c_n(\omega) \in \mathbf{H}^{2 n}(B , \mathbb{Z})$, for a suitable choice of orientation of $\omega$.
%\item $e(\xi)=-e(\xi)$ if the fibers of $\xi$ have odd dimension.
%\item $e(\xi)=0$ if $\xi$ has a nowhere-zero section.
%\end{enumerate}

%We recall the relations of Pontryagin classes with Stiefel-Whitney and Euler classes.

%\begin{enumerate}
%\item For an $n$-plane bundle $\xi$ we have that $p_i(\xi) \mapsto w_{2 i}(\xi)^2$ under the coefficient homomorphism $\mathbf{H}^{4 i}(B(\xi) , \mathbb{Z}) \rightarrow \mathbf{H}^{4 i}(B(\xi)  , \mathbb{Z}_2)$;
%\item for an orientable real $2n$-plane bundle $\xi$ with Euler class $e(\xi) \in \mathbf{H}^{2 n}(B ,\mathbb{Z}),$ $p_n(\xi)=e(\xi)^2$.
%\end{enumerate}

\subsection{Cohomology of $\BSOn$}

In this subsection we describe the mod $p$ cohomology of the classifying space for oriented $n$-plane bundles $\BSOn$.

We recall that the mod 2 cohomology of $\BSOn$ can be computed as follows, \textit{cfr.} \cite[Theorem 12.4]{MilnorStasheff}.
\begin{pro}\label{p:BSOn_mod2}
    The cohomology $\mathbf{H}^*(\BSOn,\mathbb{Z}_2)$ is a polynomial algebra over $\mathbb{Z}_2$, freely generated by the Stiefel-Whitney classes $w_2(\univob), \dots, w_n(\univob).$
\end{pro}

The cohomology ring of $\BSOn$ with coefficients in an odd prime $p$ has the following structure, \textit{cfr.} \cite[Theorem 15.9]{MilnorStasheff}.

\begin{pro}\label{p:BSOn_modp}
If $\Lambda$ is an integral domain containing $1/2$, then the cohomology ring $$\mathbf{H}^*(\mathbf{BSO}(2n+1), \Lambda)$$ is a polynomial ring over $\Lambda$ generated by the Pontrjagin classes $p_1(\widetilde{\gamma}^{2n+1}), \ldots, p_n(\widetilde{\gamma}^{2n+1}).$ Similarly, $$\mathbf{H}^*(\mathbf{BSO}(2n), \Lambda)$$ is a polynomial ring over $\Lambda$ generated by the Pontrjagin classes $p_1(\widetilde{\gamma}^{2n}), \ldots, p_{n-1}(\widetilde{\gamma}^{2n})$ and the Euler class $e(\widetilde{\gamma}^{2n})$.
\end{pro} 

That is, for every value of $n$, even or odd, the ring $\mathbf{H}^*(\BSOn, \Lambda)$ is generated by the characteristic classes $p_1, \ldots, p_{\lfloor n / 2 \rfloor}$, and $e$. These generators are subject only to the relations $e = 0$ for $n$ odd and $e^2=p_{n/2}$ for $n$ even.

\newpage
\section{List of symbols}\label{s:legenda}
Here we list some notations of the article:
\vspace{0.5cm}

\begin{tabular}{ll}

    %$\mathcal{H}^k(A)$ & $k$-dimensional Hausdorff measure of the set $A$;\\

    $\mathcal{K}^j$ & $j$-skeleton of the triangulation $\mathcal{K}$;\\

    $B_{\delta}(A)$ & $\delta$-neighborhood of $A$, \textit{i.e.} $\{x: \dist\{x,A\}< \delta\}$;\\

    $V_{\delta}(\mathcal{K}^j)$ & neighborhood of $\mathcal{K}^j$ defined in Subsection \ref{s:subsmoothingneighborhood};\\

    $U_{\delta}(\mathcal{K}^j)$ & neighborhood of $\mathcal{K}^j$ with smooth boundary defined in Lemma \ref{l:spigoli-allisciati};\\

    $\supp(T)$ & support of the current $T$;\\

    %S(X)$ & suspension of the space $X$;\\

    %$\nu_{\Sigma}$ & Riemannian normal bundle to $\Sigma$ in $\mathcal{M}$;\\

    %$T \nu_{\Sigma}$ & Thom complex of the normal bundle $\nu_{\Sigma}$;\\

    %$c: \mathcal{M} \rightarrow T \nu_{\Sigma}$ & Thom's collapsing map;\\

    $\mathbf{BSO}(n)$ & classifying space for oriented $n$-plane bundles;\\

    $\widetilde{\gamma}^n$ & universal oriented $n$-plane bundle over $\mathbf{BSO}(n)$;\\

    $T(\widetilde{\gamma}^n)$ & Thom space of the universal oriented $n$-plane bundle;\\

    $\mathbf{K}(\pi, n)$ & Eilenberg-MacLane space of type $(\pi,n)$;\\

    $Sq^i$ & Steenrod squares;\\

    $\mathcal{P}^i_p$ & Steenrod reduced $p^{\text{th}}$ power for odd prime $p$;\\

    $\beta_p$ & Bockstein homomorphism for odd prime $p$;\\

    $w_i$ & $i^{\text{th}}$ Stiefel-Whitney class of $\widetilde{\gamma}^n$;\\

    $p_i$ & $i^{\text{th}}$ Pontrjagin class of $\widetilde{\gamma}^n$;\\

    $e$ & Euler class of $\widetilde{\gamma}^n$;\\

    $St_p^{2r(p-1)+1}$ & Thom's Steenrod powers defined after Remark \ref{r:St_pi};\\

     $\widetilde{\Omega}_r$ & $r$-dimensional oriented cobordism group.\\

\end{tabular}

\newcommand{\Addresses}{{% additional braces for segregating \footnotesize
  \bigskip
  \footnotesize

  \textsc{Department of Mathematics, Princeton University, Princeton NJ 08544, USA.}\\
  \textit{E-mail address:} \texttt{browder@math.princeton.edu}

  \textsc{Department of Mathematics, Universit\`a di Trento, 38123 Trento, Italy.}\\
  \textit{E-mail address:} \texttt{gianmarco.caldini@unitn.it}

  \textsc{School of Mathematics, Institute for Advanced Study, Princeton NJ 08540, USA.}\\
  \textit{E-mail address:} \texttt{camillo.delellis@ias.edu}

}}

\Addresses


\newpage

\begin{thebibliography}

\bibitem{Allard72}
W. Allard,
\newblock {\em On the first variation of a varifold},
\newblock in ``Annals of Mathematics", \textbf{95} (3), 417--491, 1972.

\bibitem{Almgren66}
F. J. Almgren,
\newblock {\em Some interior regularity theorems for minimal surfaces and an extension of Bernstein’s theorem},
\newblock  in ``Annals of Mathematics", \textbf{84} (2), 277--292, 1966.

\bibitem{Almgren83}
F. J. Almgren,
\newblock {\em Q-valued functions minimizing Dirichlet's integral and the regularity of area minimizing rectifiable currents up to codimension two},
\newblock in ``Bulletin of the American Mathematical Society (N.S.)", \textbf{8} (2), 327--328, 1983.



\bibitem{Almgren90-93}
F. J. Almgren,
\newblock {\em Questions and Answers about Area-Minimizing Surfaces and Geometric Measure Theory}, proceedings of Differential Geometry: Partial Differential Equations on Manifolds  (Los Angeles, CA July 8-28 1990)
\newblock in ``Proceedings of Symposia in Pure Mathematics", \textbf{54} (1), 255--259, American Mathematical Society, Providence 1993.

\bibitem{Almgren00}
F. J. Almgren,
\newblock {\em Almgren’s Big Regularity Paper}, World Scientific Monograph Series
in Mathematics,
\newblock \textbf{1}, World Scientific Publishing Co. Inc., River Edge 2000.

\bibitem{AlmgrenBrowder88-91}
F. J. Almgren, W. Browder,
\newblock {\em Homotopy with holes and minimal surfaces}, proceedings of Differential Geometry: A Symposium in Honour of Manfredo Do Carmo (Rio de Janeiro, Brazil August 1988)
\newblock in ``Pitman monographs and surveys in pure and applied mathematics", \textbf{52}, 15--24, Longman Scientific \& Technical, Essex 1991.


%\bibitem{AlmgrenBrowderLieb}
%F. J. Almgren, W. Browder, E. H. Lieb,
%\newblock {\em Co-area, liquid crystals, and minimal surfaces}, proceedings of 7th symposium on
%differential geometry and differential equations (Tianjin, China June 23 - July 5 1986)
%\newblock in ``Lecture Notes in Mathematics", \textbf{1306}, 1--22, Springer, Berlin-Heidelberg 1988.

%\bibitem{AlmgrenThurston}
%F. J. Almgren, W. P. Thurston,
%\newblock {\em Examples of unknotted curves which bound only surfaces of high genus within their convex hulls},
%\newblock  in ``Annals of Mathematics", \textbf{105} (3), 527--538, 1977.

\bibitem{Benedetti}
R. Benedetti,
\newblock {\em Lectures on Differential Topology}, Graduate studies in mathematics, \textbf{218},
\newblock American Mathematical Society, Providence 2021.

\bibitem{Hankecycles}
C. Bohr, B. Hanke, D. Kotschick,
\newblock {\em Cycles, submanifolds, and structures on normal bundles},
\newblock in ``manuscripta mathematica", \textbf{108}, 483--494, 2002.


\bibitem{ButtazzoBelloni95}
G. Buttazzo, M. Belloni,
\newblock {\em A Survey on Old and Recent Results about the Gap Phenomenon in the Calculus of Variations},
\newblock in ``Mathematics and Its Applications", \textbf{331}, 1--27, 1995.

\bibitem{DeGiorgi61}
E. De Giorgi,
\newblock {\em Frontiere orientate di misura minima},
\newblock in {\em Seminario di Matematica della Scuola Normale Superiore di Pisa A.A. 1960-1961}, Editrice Tecnico Scientifica, Pisa 1961, 1--56.

\bibitem{dlsQ}
C. De Lellis and E. Spadaro,
\newblock {\em $Q$-valued functions revisited},
\newblock in ``Memoirs of the American Mathematical Society", \textbf{211} (991), 2011.

\bibitem{dls1}
C. De Lellis and E. Spadaro,
\newblock {\em Regularity of area minimizing currents I: gradient $L^{p}$ estimates},
\newblock in ``Geometric and Functional Analysis", \textbf{24}, 1831--1884, 2014.

\bibitem{dlssns}
C. De Lellis and E. Spadaro,
\newblock {\em Multiple valued functions and integral currents},
\newblock in ``Annali della Scuola Normale Superiore di Pisa - Classe di Scienze", \textbf{14} (5), 1239--1269, 2015.
  
\bibitem{dls2}
C. De Lellis and E. Spadaro,
\newblock {\em Regularity of area minimizing currents II: center manifold},
\newblock in ``Annals of Mathematics", \textbf{183} (2), 499--575, 2016.
 
  \bibitem{dls3}
C. De Lellis and E. Spadaro,
\newblock {\em Regularity of area minimizing currents III: blow-up},
\newblock in ``Annals of Mathematics", \textbf{183} (2), 577--617, 2016.


\bibitem{DubrovinFomenkoNovikovIII}
B.A. Dubrovin, A.T. Fomenko, S.P. Novikov,
\newblock {\em Modern Geometry} - {\em Methods and Applications III}, Graduate Texts in Mathematics, \textbf{124},
\newblock Springer, Berlin-Heidelberg 1990.



\bibitem{Federer59}
H. Federer,
\newblock {\em An approximation theorem concerning currents of finite mass}, proceedings of Sixty-fourth Summer Meeting and Thirty-eight Colloquium (Salt Lake City, Utah September 1-4 1959)
\newblock in ``Notices of the American Mathematical Society", \textbf{6} (4), 326, American Mathematical Society, Providence 1959.

\bibitem{Federerbook}
H. Federer,
\newblock {\em Geometric measure theory}, Classics in Mathematics,
\newblock Springer, Berlin-Heidelberg 1996.

\bibitem{FedererFleming60}
H. Federer and W. Fleming,
\newblock {\em Normal and Integral Currents},
\newblock in ``Annals of Mathematics", \textbf{72} (3), 458--520, 1960.


\bibitem{GoreskyMacPherson}
M. Goresky , R. MacPherson,
\newblock {\em Stratified Morse Theory}, Ergebnisse der Mathematik und ihrer Grenzgebiete, \textbf{14},
\newblock Springer, Berlin-Heidelberg 1988.


\bibitem{Jacob}
J. Lurie,
\newblock {\em Topics in Geometric Topology (18.937)},
\newblock Lecture notes available at https://www.math.ias.edu/~lurie/937.html

\bibitem{Lavrentiev27}
M. Lavrentiev,
\newblock {\em Sur quelques problèmes du calcul des variations},
\newblock in ``Annali di Matematica Pura ed Applicata", \textbf{4} (1), 7--28, 1927.

\bibitem{MilnorStasheff}
J. Milnor and J. Stasheff,
\newblock {\em Characteristic Classes}, Annals of Mathematics Studies, \textbf{76},
\newblock Princeton University Press, Princeton 1974.

%\bibitem{MunkresTriangulations}
%J. R. Munkres,
%\newblock {\em Elementary Differential Topology}, Annals of Mathematics Studies, \textbf{54}
%\newblock Princeton University Press, Princeton 1963.

\bibitem{MunkresElements}
J. R. Munkres,
\newblock {\em Elements of Algebraic Topology},
\newblock Addison-Wesley, Menlo Park 1984.

%\bibitem{Quinn}
%F. Quinn, Frank,
%\newblock {\em Lectures on controlled topology: mapping cylinder neighborhoods},
%\newblock in {\em Topology of high-dimensional manifolds}, {ICTP Lect. Notes}, {9}, {461--489}, {Abdus Salam Int. Cent. Theoret. Phys., Trieste}, {2002}.



\bibitem{Serre53}
J. P. Serre,
\newblock {\em Cohomologie modulo 2 des complexes d’Eilenberg-MacLane},
\newblock in ``Commentarii Mathematici Helvetici", \textbf{27}, 198--232, 1953.


\bibitem{Simonbook}
L. Simon,
\newblock {\em Lectures on geometric measure theory}, Proceedings of the Centre for Mathematical Analysis, \textbf{3},
\newblock Australian National University, Canberra 1984.

\bibitem{Spanier}
E. H. Spanier,
\newblock {\em Algebraic Topology},
\newblock McGraw-Hill, New York 1966.


\bibitem{SteenrodEpstein}
N. E. Steenrod and D. B. A. Epstein,
\newblock {\em Cohomology operations}, Annals of Mathematics Studies, \textbf{50},
\newblock Princeton University Press, Princeton 1962.


\bibitem{SullivanLiverpool}
D. Sullivan,
\newblock {\em Singularities in spaces}, proceedings of Liverpool Singularities Symposium II (Liverpool, UK September 1969 - August 1970)
\newblock in ``Lecture notes in mathematics", \textbf{209}, 196--206, Springer, Berlin-Heidelberg 1971.



\bibitem{Switzer}
R. M. Switzer,
\newblock {\em Algebraic Topology - Homology and Homotopy}, Classics in Mathematics,
\newblock Springer, Berlin-Heidelberg 2002.



\bibitem{MosherTangora}
R. E. Mosher and M. C. Tangora,
\newblock {\em Cohomology operations and applications in homotopy theory}, Harper's Series in Modern Mathematics,
\newblock Harper \& Row, New York 1968.

\bibitem{Thomdottorato}
R. Thom,
\newblock {\em Espaces fibrés en sphères et carrés de Steenrod},
\newblock in ``Annales scientifiques de l'École normale supérieure", \textbf{69}, 109--182, 1952.


\bibitem{Thom54}
R. Thom,
\newblock {\em Quelques propriétés globales des variétés
différentiables},
\newblock in ``Commentarii Mathematici Helvetici", \textbf{28}, 17--86, 1954.

\bibitem{Thom69}
R. Thom,
\newblock {\em Ensembles et morphismes stratifiés},
\newblock in ``Bulletin of the American Mathematical Society", \textbf{75}, 240--284, 1969.


\bibitem{Wall}
C. T. C. Wall,
\newblock {\em Differential Topology}, Cambridge studies in advanced mathematics, \textbf{156},
\newblock Cambridge University Press, Cambridge 2016.


\bibitem{Whitehead}
J. H. C. Whitehead,
\newblock {\em On $C^1$-complexes},
\newblock in ``Annals of Mathematics", \textbf{41} (4), 809--824, 1940.




\end{thebibliography}
\end{document}